\documentclass[a4paper,12pt]{article}
\usepackage{amsmath, amsfonts,amsthm, amssymb}
\usepackage{array}
\usepackage{graphics}
\usepackage{epsfig}
\usepackage{color}
\usepackage{fancyhdr}
\usepackage{moreverb}
\usepackage{sectsty}
\usepackage{titlesec}
\usepackage[english]{babel}
\usepackage[T1]{fontenc}
\usepackage{pstricks}
\usepackage{ae}
\usepackage{aecompl}
\usepackage[nottoc, notlof, notlot]{tocbibind}
\usepackage{listings}

\newtheorem{theorem}{Theorem}[section]

\newtheorem{proposition}{Proposition}[section]
\newtheorem{lemma}{Lemma}[section]
\newtheorem{exo}{Exercise}[section]
\newtheorem{ex}{Example}[section]
\newtheorem{question}{Question}[section]

\title{{\em From Black-Scholes and
    Dupire formulae to last passage times of local martingales} \\ \bigskip
  Part B : The finite time horizon}
\author{Amel Bentata and Marc Yor}
\date{July, $2^{nd}$, 2008}

\author{Amel Bentata\footnote{Universit\'e Paris 6, Laboratoire de Probabilit\'es et
    Mod\`eles Al\'eatoires, CNRS-UMR 7599, 16, rue Clisson, 75013 Paris Cedex,
    France} and Marc Yor\footnote{Universit\'e Paris 6, Laboratoire de Probabilit\'es et
    Mod\`eles Al\'eatoires, CNRS-UMR 7599, 16, rue Clisson, 75013 Paris Cedex,
    France-Institut Universitaire de France}}

\begin{document}

\setcounter{section}{5}

\maketitle
\begin{enumerate}
\item[1.] 
These notes are the second half of the contents of the course given by the
  second author at the Bachelier Seminar (8-15-22 February 2008) at IHP.
  They also correspond to topics studied by the first author for her Ph.D.thesis.
\item[2.]

Unlike Part A of the course (\cite{amel0}), this document still raises a number of questions, pertaining to various extensions of the classical Black-Scholes formula.
\item[3.]

Comments are welcome and may be addressed to : \\bentata@clipper.ens.fr.
\end{enumerate}
\newpage
\tableofcontents
\newpage

\section*{A rough description of Part B of the course}
In Part B of the Course, we further our discussion of relations bewteen Black-Scholes type formulae and distributions of last times, begun in Notes 1 to 5 of Part A (\cite{amel0}).

In Note $6$, following a suggestion by J. Akahori, we consider, instead of
the last passage times $\mathcal{G}_K$, which ``look into the filtration up
to time $\infty$'', the times : $g_{l}^{(\nu)}(t)=\sup\{s\leq t,B_s^{(\nu)}=l\}$\footnote{In general, we use capital letters such as $K$ for ``geometric type'' objects, and small letters such as $l$ for items related to Brownian motion with drift, with the correspondence : $l=\ln{K}$.} which present the advantage, although they are not stopping times, to look into the filtration only up to the finite horizon $t$. Computing the associated Azéma's supermartingale : $Z_{s,t}=\mathbb{P}\left(g_{l}^{(\nu)}(t)> s|\mathcal{F}_s\right)$, for $s<t$, leads us to study (super-)martingales with respect to the two-parameter filtration $\mathcal{F}_{s,t}=\sigma\{B_u^{(\nu)}, u\leq s, B_v^{(\nu)}, v\geq t\}$, $0<s<t<\infty$.

In Note $7$, we no longer make the assumption that $(M_t)$ is continuous, and we discuss which of the results from the previous notes still hold, or how they may be modified. In particular, some of our discussion applies to the Esscher martingales : $\mathcal{E}_t^{\psi}=\exp{\left(\lambda X_t-t\psi(\lambda)\right)}$, $t\geq 0$, associated with a Lévy process $(X_t)$ with some exponential moments.

In Note $8$, we discuss a recent result by P. Carr, C. Ewald and Y. Xiao asserting that $\mathbb{E}\left[\left(\frac{1}{T}A_T-K\right)^+\right]$ is increasing in $T$, when $A_t=\int_0^tds\,\exp{\left(B_s-\frac{s}{2}\right)}$; clearly, our previous discussion does not apply there.

In Note $9$, we examine how careful one should be when changing ``puts'' into ``calls'', and examining the modifications one should make in our previous discussions.

Note 10 concludes : a number of our results may be stated in quite
some generality, at the cost of introducing certain conditional
expectations, which, in the Markovian context, are easily computed,
but in the general martingale framework, bring some fundamental
``unknowns'' into the discussion. Thus, this case study may be seen as a further illustration of the passage from a Markovian discussion to a ``general theory of stochastic processes'' type discussion.

Finally, an appendix, presents a number of developments/ solutions for some of the exercises suggested in the different notes of both Part A and Part B.

Throughout the text, we may refer to some of the results of Part A; when we do so we write e.g : $A-(65)$ for the equation $(65)$ in Part A.
\newpage

\section{Note 6 : Working under a finite horizon}
This note aims to counterbalance the difficulty of bringing into the
picture quantities such as $\mathcal{G}_K=\sup\{t,M_t=K\}$ which take into
account the whole history of $M$ (or $(\mathcal{F}_t)$).

Thus, we are now interested in finding formulae up to finite maturity,
which may replace our previous formula :
\begin{equation}
\mathbb{P}\left(\mathcal{G}_K\leq t|\mathcal{F}_t\right)=\left(1-\frac{M_t}{K}\right)^{+},
\end{equation}
when $M\in\mathcal{M}_0^{+}$, or even the more general\footnote{We often
  write $\mathbb{E}\left[\Lambda; X\right]$ for $\mathbb{E}\left[1_{\Lambda}X\right]$} :
\begin{equation}
\mathbb{E}\left[\left(\mathcal{G}_K\leq t\right); \left(K-M_\infty\right)^{+})\right]=\mathbb{E}\left[\left(K-M_t\right)^{+}\right]
\end{equation}
The following developments are most recent\footnote{They were very much motivated by some computations of Akahori-Imamura-Yano \cite{yuri}, which Madan-Roynette-Yor \cite{madyorroy4} have later shown to match formula (\ref{eq84}) below.}, therefore we shall study
only the particular case where :
$M_t=\mathcal{E}_t=\exp{\left(B_t-\frac{t}{2}\right)}$, and $\mathcal{G}_K$ is
replaced by $\mathcal{G}_K(t)=\sup\{s\leq t,\mathcal{E}_s=K\}$.
\subsection{Past-future (sub)martingales}
 Rather than starting with some ``explicit'' computation about $\mathcal{G}_K(t)$, it seems of interest (see below !) to introduce the past-future filtration of Brownian motion :
\[\mathcal{F}_{s,t}=\sigma\{B_u,u\leq s,B_v,v\geq t\},\]
indexed by pairs $(s,t)$, with $0\leq s\leq t<\infty$.

Note that, for $s=t$, $\mathcal{F}_{s,s}=\mathcal{F}_\infty$, and that, if $[s,t]\subset[s',t']$, then : $\mathcal{F}_{s,t}\supseteq\mathcal{F}_{s',t'}$.

In this note, we shall naturally deal with :
\renewcommand{\theenumi}{\roman{enumi}.}
\renewcommand{\labelenumi}{\theenumi}
\begin{enumerate}
\item $(\mathcal{F}_{s,t})$ positive submartingales, i.e : $\mathbb{R}^{+}$ valued, $(\mathcal{F}_{s,t})$ adapted processes $(\Sigma_{s,t})$ such that : 
\begin{equation}\label{6eq1}
\mathbb{E}\left[\Sigma_{s,t}|\mathcal{F}_{s',t'}\right] \geq \Sigma_{s',t'};
\end{equation}
\item $(\mathcal{F}_{s,t})$ positive martingales, i.e :
\begin{equation}\label{6eq2}
\mathbb{E}\left[M_{s,t}|\mathcal{F}_{s',t'}\right]=M_{s',t'}.
\end{equation} 
\end{enumerate}
Here are some particular examples of such processes :
let $\Gamma$ be a Borel set in $\mathbb{R}$, and define :
\begin{equation}
\Sigma_{s,t}^{\Gamma}=\mathbb{P}\left(\left\{\forall u\in (s,t),B_u\in\Gamma\right\}|\mathcal{F}_{s,t}\right),
\end{equation} 
and the sets :
\[A_{s,t}^{\Gamma}=\{\forall u\in(s,t),B_u\in \Gamma\}.\]
They satisfy : $A_{s,t}^{\Gamma}\supset A_{s',t'}^{\Gamma} $,
which implies that $\Sigma_{s,t}^{\Gamma}$ satisfies (\ref{6eq1}).

Moreover, from the Markov property of Brownian motion, there exists a function $f^\Gamma(s,t;x,y)$ such that :
\begin{equation}\label{6eq4}
\Sigma_{s,t}^{\Gamma}=f^\Gamma(s,t;B_s,B_t).
\end{equation}
We shall call such functions $f$ (depending on the 4 arguments) past-future subharmonic functions ($PFS$-functions).
Naturally, we may also be interested in past-future harmonic functions $h(s,t;x,y)$ ($PFH$-functions) such that :
\[\Sigma_{s,t}^{h}=h(s,t;B_s,B_t) \quad \mathrm{is\:a}\:
(\mathcal{F}_{s,t})\mathrm{-martingale}.\]
Thanks to basic properties of Brownian motion, i.e : stability by time inversion, that is : if $(B_t,t\geq 0)$ is a Brownian motion starting from $0$, then :
\[(\hat{B}_t=tB_{1/t},t> 0)\]
is also a Brownian motion (which extends to 0, by continuity), and so on, the
following Proposition is easily obtained :
\begin{proposition}\label{timeinver}
\renewcommand{\theenumi}{\roman{enumi}.}
\renewcommand{\labelenumi}{\theenumi}
\begin{enumerate}
\item
If $f(\equiv f(s,t;x,y))$ is $PFH$, resp $PFS$, then $\hat{f}$ is also
$PFH$, resp $PFS$, where :
\begin{equation}
\hat{f}(s,t;x,y)=f(\frac{1}{t},\frac{1}{s},\frac{y}{t},\frac{x}{s})\quad (0<s<t;x,y\in\mathbb{R})
\end{equation}
We also note that the application $f\to\hat{f}$ is involutive, i.e : \[\hat{\hat{f}}=f.\]
\item
If $f$ is $PFH$, then, for any $\nu$, $\lambda\in\mathbb{R}$ :
\begin{equation}
f_{\nu,\,l}(s,t;x,y)\equiv f(s,t;x+\nu s+l,y+\nu t+l)
\end{equation}
is $PFH$.
\item
If $f$ is $PFH$, then for any $a>0$,
\begin{equation}
f_{(a)}(s,t;x,y)\equiv f(a^2s,a^2t;ax;ay)
\end{equation}
is also $PFH$
\end{enumerate}
\end{proposition}
\begin{proof}
We leave the proof to the reader.
\end{proof}

\subsection{The case of Brownian motion with drift}
\subsubsection{}
Here are some particular examples of the previous discussion, which are closely connected with our random times :
\[g_l^{(\nu)}(t)=\sup\{s\leq t, B_s^{(\nu)}=l\},\]
where : \[B_s^{(\nu)}=B_s+\nu s.\]

Indeed, for $s\leq t$ :
\begin{equation}
\left(g_l^{(\nu)}(t)\leq s\right)=A_{s,t}^{+} + A_{s,t}^{-},
\end{equation}
where :
\begin{equation}
  \begin{cases}
    A_{s,t}^{+} &= \{\forall u \in (s,t), B_u^{(\nu)} >l\}\\
    A_{s,t}^{-} &= \{\forall u \in (s,t), B_u^{(\nu)} <l\}.
  \end{cases}
\end{equation}
Therefore,
\begin{eqnarray*}
  \Sigma_{s,t}^{(l,\,\nu)}&\overset{\underset{\mathrm{def}}{}}{=}&\mathbb{P}\left(g_l^{(\nu)}(t)\leq s|\mathcal{F}_{s,t}\right)\\
  &=&\Sigma_{s,t}^{+}+\Sigma_{s,t}^{-},
\end{eqnarray*}
with 
\begin{equation}
  \begin{cases}
    \Sigma_{s,t}^{+} &= \mathbb{P}\left(A_{s,t}^{+}|\mathcal{F}_{s,t}\right) \\
    \Sigma_{s,t}^{-} &= \mathbb{P}\left(A_{s,t}^{-}|\mathcal{F}_{s,t}\right) .
  \end{cases}
\end{equation}
Our main result here is :
\begin{theorem}\label{theo6}
\begin{equation}\label{eq84}
\Sigma_{s,t}^{(l,\,\nu)}=\left(1-\exp{\left(-\frac{2}{t-s}(B_s^{(\nu)}-l)(B_t^{(\nu)}-l)\right)}\right)^{+}
\end{equation}
\end{theorem}
Thus the function :
\begin{equation}
\begin{split}
f^{(l,\,\nu)}(s,t;x,y)&=\left(1-\exp{\left(-\frac{2}{t-s}(x+\nu s-l)(y+\nu t-l)\right)}\right)^{+}\\
&=1-\exp{\left(-\frac{2}{t-s}\left((x+\nu s-l)(y+\nu t-l)\right)^+\right)}
\end{split}
\end{equation}
is $PFS$.

This leads us naturally to the next theorem :

\begin{theorem}\label{theo7}
The function :
\begin{equation}
h^{(l,\nu)}(s,t;x,y)=\exp{\left(-\frac{2}{t-s}(x+\nu s-l)(y+\nu t-l)\right)}
\end{equation}
is $PFH$.
\end{theorem}
The sequel of this note consists in proving both theorems.

As an illustration of both theorems, and of Proposition \ref{timeinver},
an elementary computation shows that :
\begin{equation}
\widehat{h^{(l,\,\nu)}}= h^{(\nu,\,l)}; \widehat{f^{(l,\,\nu)}}= f^{(\nu,\,l)}
\end{equation}

\subsubsection{A connection with Section 4.3 (Part A)}
We now interpret our result (\ref{eq84}) in terms of the martingale : $(\mathcal{E}_u,u\leq t)$, which terminates at $\mathcal{E}_t>0$. Now, with the same notation as in Section 4.3 (Part A), the Azéma supermartingale : $(Z_s^K,s\leq t)$ is : 
\[Z_s^K=1-\Sigma_{s,t}^{(l,\,\nu)},\]
with $l=\log{(K)}$ and $\nu=-\frac{1}{2}$.\\
Hence : 
\[Z_s^K=\exp{\left(-\frac{2}{(t-s)}\left((B_s^{(-1/2)}-l)(B_t^{(-1/2)}-l)\right)^+\right)}.\]
We consider the easiest case : $K=1$, hence : $l=0$, so that :
\[\mathcal{G}_1(t)=\sup\{s\leq t,B_s-\frac{s}{2}=0\},\]
and we wish to compute :
\begin{equation}
\mathbb{E}\left[f\left(\exp{(B_t-\frac{t}{2})}\right)|\mathcal{G}_1(t)\right]=\frac{\mathbb{E}\left[f\left(\exp{(B_t)}\right)\exp{(-\frac{B_t}{2})}|g_t\right]}{\mathbb{E}\left[\exp{(-\frac{B_t}{2})}|g_t\right]},
\end{equation}
where : $g_t=\sup\{u\leq t, B_u=0\}$.\\
Now, we use the fact that : $B_t=B_{g_t+(t-g_t)}=\sqrt{t-g_t}\,\tilde{m}_1$ where $\tilde{m}_1=\epsilon\sqrt{2\mathbf{e}}$, and $g_t$, $\epsilon$, $\mathbf{e}$ are independent. $\epsilon$ is symmetric Bernoulli, and $\mathbf{e}$ is a standard exponential variable.\\
Since :
\[\mathbb{P}\left(\tilde{m}_1\in d\mu\right)=\frac{1}{2}|\mu|e^{-\mu^2/2}d\mu,\:(\mu\in\mathbb{R}),\]
we obtain with $\lambda=\sqrt{t-g_t}$ :
\begin{equation}
\mathbb{E}\left[f\left(\mathcal{E}_t\right)|\mathcal{G}_1(t)\right]=
\frac{\int_{-\infty}^\infty\frac{d\mu}{2}|\mu|e^{-\mu^2/2}e^{-\lambda\mu/2}f(e^{\lambda\mu})}
{\int_{-\infty}^\infty\frac{d\mu}{2}|\mu|e^{-\mu^2/2}e^{-\lambda\mu/2}},
\end{equation}
which provides us with an explicit form of the law $\nu_1(dm)$ defined in $Part A-(125)$. From there, all quantities found in Subsection $A-4.3$ can be made explicit.

\subsection{Proof of Theorem \ref{theo7}}
It will follow from the next proposition
\begin{proposition}\label{prop8}
A regular function $h(s,t;x,y)$ is $PFH$ if and only if it satisfies the following system :
\begin{equation}\label{syst-}\tag{$-$}
h'_s(s,t;x,y)+\frac{y-x}{t-s}h'_x(s,t;x,y)+\frac{1}{2}h''_{x^2}(s,t;x,y)=0
\end{equation}
\begin{equation}\label{syst+}\tag{$+$}
-h'_t(s,t;x,y)-\frac{y-x}{t-s}h'_y(s,t;x,y)+\frac{1}{2}h''_{y^2}(s,t;x,y)=0
\end{equation}
\end{proposition}

\begin{enumerate}
\item[$\bullet$]The proof of this proposition hinges on the following lemma :
  \begin{lemma}\label{martin}
   Let \[M_{s,t}^f=f(s,t;B_s,B_t),\]
    for $f$ a regular function; it is a past-future
    martingale if and only if :
    \begin{enumerate}
    \item[a)] for fixed t, $\{(M_{s,t}^f)_s, s<t\}$ is a $(\mathcal{F}_{s}^{(t)}, s<t)$-martingale, where :
      \[\mathcal{F}_s^{(t)}=\mathcal{F}_s \vee \sigma(B_t),\]
    \item[b)] for fixed s, $\{(M_{s,t}^f)_t, t>s\}$ is a $(^{(s)}\!\mathcal{F}_t^{(+)},t>s)$-martingale where :
      \[^{(s)}\!\mathcal{F}_t^{(+)}=\sigma(B_s) \vee \mathcal{F}_t^{(+)},\quad
      \mathrm{and}\quad \mathcal{F}_t^{(+)}=\sigma(B_v,v\geq t)\]
    \end{enumerate}
\end{lemma}  
  \begin{proof}
      Note that $(M_{s,t}^f)$ is an $(\mathcal{F}_{s,t})$-martingale if and only if : \\for every $\phi_{s'}\in b(\mathcal{F}_{s'})$, $\psi_{t'}\in b(\mathcal{F}_{t'}^{(+)})$, with $s'<s<t<t'$, one has :
      \begin{equation}\label{retoile}
        \mathbb{E}\left[\phi_{s'}M_{s,t}^f\psi_{t'}\right]=\mathbb{E}\left[\phi_{s'}M_{s',t'}^f\psi_{t'}\right].
      \end{equation}
      We now prove :
      \begin{enumerate}
      \item[1.] (\ref{retoile})$\Longrightarrow$ a) and b) :\\
        Take $s'< s < t=t'$, and $\psi_t=g(B_t)$ for a generic bounded function $g$; this proves (\ref{retoile}) $\Longrightarrow$ a), and by past-future symmetry, (\ref{retoile}) $\Longrightarrow$ b).
      \item[2.]  a) and b) $\Longrightarrow$ (\ref{retoile}) :\\
        We wish to prove (\ref{retoile}). From the Markov property, the $LHS$ of (\ref{retoile}), say $l$, is equal to : 
        \[l=\mathbb{E}\left[\phi_{s'}M_{s,t}^f\gamma(t,B_t)\right],\]
        where \[\gamma(t,B_t)=\mathbb{E}\left[\psi_{t'}|\mathcal{F}_t\right]=
        \mathbb{E}\left[\psi_{t'}|B_t\right].\]
        From a), we now get :
        \[l=\mathbb{E}\left[\phi_{s'}M_{s,t}^f\gamma(t,B_t)\right]=\mathbb{E}\left[\phi_{s'}M_{s',t}^f\psi_{t'}\right].\]
        Again, from the Markov property, we get :
        \[l=\mathbb{E}\left[\beta(s',B_{s'})M_{s',t}^f\psi_{t'}\right],\]
        where : 
        \[\beta(s',B_{s'})=\mathbb{E}\left[\phi_{s'}|B_{s'}\right].\] 
        Now, using b), we get :
        \[l=\mathbb{E}\left[\beta(s',B_{s'})M_{s',t}^f\psi_{t'}\right]=\mathbb{E}\left[\phi_{s'}M_{s',t}^f\psi_{t'}\right].\]
        The lemma is proven\footnote{We remark that these arguments only make use of the Markov property and not of the Brownian framework. A further discussion is provided in \cite{amel}.}.
      \end{enumerate}
    \end{proof}
  \item[$\bullet$]We now give a proof of Proposition \ref{prop8} :
\begin{proof}
    For fixed $t>0$, $(B_s, s\leq t)$ is a $(\mathcal{F}_s^{(t)}, s\leq t)$ semimartingale, with :
    \begin{equation}
      B_s=\beta_s^{(t)}+\int_0^s\,du\,\frac{B_t-B_u}{t-u},
    \end{equation}
    and $(\beta_s^{(t)},s\leq t)$ a $(\mathcal{F}_s^{(t)},s\leq t)$ Brownian motion
    (see \cite{jeulyor2}; also \cite{zurich}).

    We now apply Itô's formula to $(M_{s,t}^f,s\leq t)$ in the filtration $(\mathcal{F}_s^{(t)},s\leq t)$.
    We obtain for $u<s<t$ :
    \begin{equation}
      \begin{split}
        M_{s,t}=M_{u,t}&+\int_u^s\,dh\,f'_s(h,t;B_h,B_t)\\
        &+\int_u^s(d\beta_h^{(t)}+dh\,\frac{B_t-B_h}{t-h})\,f'_x(h,t;B_h,B_t)\\
        &+ \frac{1}{2}\int_u^s\,dh\,f''_{x^2}(h,t;B_h,B_t).
      \end{split}
    \end{equation}
    Hence, the martingale property holds if and only if :
    \begin{equation}
      f'_s(s,t;x,y)+\frac{y-x}{t-s}f'_x(s,t;x,y)+\frac{1}{2}f''_{x^2}(s,t;x,y)=0.
    \end{equation}
Thus, we have obtained $(-)$.
We might obtain $(+)$ using some similar arguments;
however a time inversion argument is much quicker and reduces the
obtention of $(+)$ to $(-)$. 
\end{proof}
\item[$\bullet$]Elementary computations now show that $h^{(l,\,\nu)}$ is $PFH$; that is we have proven Theorem \ref{theo7}.
\end{enumerate}

\begin{exo}\label{sol1}
Consider $h^{(l,\,\nu)}$ as a generating function, depending on the $2$ parameters $l,\nu$.
We write :
\begin{eqnarray*}
  h^{(l,\,\nu)}(s,t;x,y)&=&\exp{\left(-\frac{2}{t-s}(x+\nu s-l)(y+\nu t-l)\right)}\\
&=&\sum_{p=0}^{\infty}\sum_{q=0}^{\infty} l^p\nu^q\: H_{p,q}(s,t;x,y).
\end{eqnarray*} 
Prove that for any $p,q$, $H_{p,q}$ is $PFH$, and give an expression
for $H_{p,q}$. 
\end{exo}
\begin{exo}\label{sol2}
Introduce the 5 variables Hermite polynomials $\mathcal{H}_{p,q}(a,b,c,d,f)$
defined by the generating function :
\[
\exp{(al+b\nu-\frac{c\nu^2}{2}-\frac{dl^2}{2}+fl\nu)}=\sum_{p=0,\,q=0}^\infty l^p\nu^q\,\mathcal{H}_{p,\,q}(a,b,c,d,f)
.\]
Write $H_{p,\,q}$ in terms of $\mathcal{H}_{p,\,q}$.
\end{exo}
Solutions to Exercises \ref{sol1} and \ref{sol2} are given in the Appendix.
\begin{exo}
(A second family of $PFH$ functions). Check that, for any
$a\in\mathbb{R}$,
\begin{equation}
h^{(a)}(s,t;x,y)=\exp{\left(a\frac{y-x}{t-s}-\frac{a^2}{2}\frac{1}{t-s}\right)}
\end{equation}
is a $PFH$ function.\\
Hint: recall that, if 
\[B_{[s,t]}=\frac{B_t-B_s}{t-s},\]
then :
\[\mathbb{E}\left[B_{[s,t]}|\mathcal{F}_{s',t'}\right]=B_{[s',t']}.\]
This property of $(B_t)$ is the Harness property, which holds more generally
for any integrable Lévy process. More generally, see Mansuy-Yor (\cite{manyor2}) on Harnesses. We refer to  \cite{chaumont} (exercise $6.19$), \cite{jacod}, \cite{williams3}, \cite{williams4}.
\end{exo}
\begin{exo}
Give other examples of $PFH$ functions for other Lévy processes.\\
An excellent candidate is the Gamma process $(\gamma_t, t\geq 0)$, since the Harness property in this case :
\[\mathbb{E}\left[\gamma_{[s,t]}|\mathcal{F}_{s',t'}\right]=\gamma_{[s',t']},\]
is reinforced as follows :
\[\frac{\gamma_t-\gamma_s}{\gamma_{t'}-\gamma_{s'}}\:\mathrm{is\:independent\:from}\:\mathcal{F}_{s',t'},\]
and is distributed as :
\[\beta(t-s;(t'-s')-(t-s))\]
where $\beta(a;b)$ is the beta variable with parameters $a$ and $b$.
Thus, we have, for any $m>0$ :
\[\mathbb{E}\left[\left(\frac{\gamma_t-\gamma_s}{\gamma_{t'}-\gamma_{s'}}\right)^m|\mathcal{F}_{s',t'}\right]=\frac{\mathbb{E}\left[\left(\gamma_t-\gamma_s\right)^m\right]}{\mathbb{E}\left[\left(\gamma_{t'}-\gamma_{s'}\right)^m\right]},\]
ie : for any $m>0$ :
\[\frac{\left(\gamma_t-\gamma_s\right)^m}{\mathbb{E}\left[\left(\gamma_t-\gamma_s\right)^m\right]}=\frac{\left(\gamma_t-\gamma_s\right)^m}{\Gamma(m+(t-s))/\Gamma(t-s)}\:\mathrm{is\:a\:}(\mathcal{F}_{s,t})\:\mathrm{martingale}.\]
For more references about the nice properties of the Gamma process, see, e.g. \cite{yor3}.
\end{exo}
\begin{exo}
Give a version of Proposition \ref{prop8}, i.e : the differential
system ( $(-)$ and $(+)$ ), when Brownian motion now takes values in $\mathbb{R}^n$.\\
(Solution : $\frac{y-x}{t-s}h'_x$ is replaced by : $\frac{y-x}{t-s}\cdot\nabla_xh$,
and $h''_{x^2}$ is replaced by : $\Delta_xh$, and so on...)
\end{exo}

\subsection{Proof of Theorem \ref{theo6}}
We have :
\begin{equation}
Z_{s,t}^{(+)}=1_{\{B_s^{(\nu)}>l\}}\:\mathbb{P}\left(\inf_{\{s\leq u\leq t\}}\left(B_u^{(\nu)}-B_s^{(\nu)}\right)>l-B_s^{(\nu)}|B_t^{(\nu)}-B_s^{(\nu)}\right).
\end{equation}
Thus, we need to compute, for $l$, and $\lambda=l-B_s^{(\nu)}$ :
\begin{equation}\label{etoile6.4}
\begin{split}
&\mathbb{P}\left(\inf_{\{s\leq u\leq
    t\}}\left(B_u^{(\nu)}-B_s^{(\nu)}\right)>l-B_s^{(\nu)}|B_t^{(\nu)}-B_s^{(\nu)}=m\right)\\
&=\mathbb{P}\left(T_\lambda>(t-s)|B_{(t-s)}=m\right),
\end{split}
\end{equation}
since :
\renewcommand{\theenumi}{\roman{enumi}.}
\renewcommand{\labelenumi}{\theenumi}
\begin{enumerate}
\item the quantity in (\ref{etoile6.4}) clearly depends only on $(t-s)$ and
  not on the pair $(s,t)$,
\item the law of the Bridge from $a$ to $b$, over the time interval
  $[0,\alpha]$ of a Brownian motion with drift $\nu$ does not depend
  on $\nu$. 
\end{enumerate}
Thus, from (\ref{etoile6.4}), we need to compute :
\begin{equation}\label{etoilebis6.4}
\mathbb{P}\left(T_\lambda>(t-s)|B_{(t-s)}=m\right)=1-\mathbb{P}\left(T_\lambda<(t-s)|B_{(t-s)}=m\right).
\end{equation}
Let us denote by $\mathbb{P}_{0\to m}^{(u)}$ the law of the Brownian
bridge of length $u$, starting at $0$, ending at $m$. Then, there is
the result :
\begin{equation}\label{point6.4}
\mathbb{P}_{0\to m}^{(u)}\left(T_\lambda<u\right)=\exp{\left(-\frac{(2\lambda(\lambda-m))^{+}}{u}\right)}
\end{equation}
\underline{Comment :} We learnt from G.Pagès \cite{pages} that the Brownian Bridge method to estimate the error due to discretization for path dependent options relies upon formula (\ref{point6.4}), which is also proven in Section 7.7 of \cite{pages}.\bigskip

We give two proofs of (\ref{point6.4}) :
\begin{enumerate}
\item[$\bullet$]\textbf{A first proof via the reflection principle :}\\
\renewcommand{\theenumi}{\roman{enumi}.}
\renewcommand{\labelenumi}{\theenumi}
\begin{enumerate}
\item By scaling, we can restrict ourselves to $u=1$;
\item In \cite{yor2}, it is remarked, following \cite{sesh}, that if
  $S_t=\sup_{u\leq t}B_u$, then $2S_1(S_1-B_1)$ is independent from
  $B_1$, and satisfies :
\[2S_1(S_1-B_1)\overset{\underset{\mathrm{law}}{}}{=}\mathbf{e},\]
where $\mathbf{e}$ is a standard exponential variable.
Thus :
\begin{eqnarray*}
\mathbb{P}_{0\to m}^{(1)}\left(T_\lambda<1\right)&=&\mathbb{P}_{0\to
  m}^{(1)}\left(S_1>\lambda\right)\\
&=&\mathbb{P}_{0\to m}^{(1)}\left(2S_1(S_1-B_1)>2\lambda(\lambda-m)\right)\\
&=&\exp{\left(-2\lambda(\lambda-m)\right)}
\end{eqnarray*}
(here, we have assumed a priori : $\lambda(\lambda-m)>0$.)
\end{enumerate}
\item[$\bullet$]\textbf{A second proof via Doob's maximal identity :}
In order to show :
\begin{equation}\label{united}
p_{\lambda,\,m}\equiv\mathbb{P}\left[\sup_{u\leq
    1}B_u>\lambda|B_1=m\right]=\exp{\left(-2\lambda(\lambda-m)\right)},
\end{equation}
$(\lambda>m,\lambda>0)$, we use time inversion :
\begin{eqnarray*}
p_{\lambda,\,m}&=&\mathbb{P}\left[\sup_{t\geq 1}B_{1/t}>\lambda|B_1=m\right]\\
&=&\mathbb{P}\left[\sup_{t\geq 1}\frac{\hat{B}_t}{t}>\lambda|\hat{B}_1=m\right]\\
&=&\mathbb{P}\left[\exists t\geq 1,\left(\hat{B}_t-\lambda t\right)>0|\hat{B}_1=m\right]\\
&=&\mathbb{P}\left[\exists u\geq 0,\left(\hat{B}_{1+u}-\lambda (1+u)\right)>0|\hat{B}_1=m\right]\\
&=&\mathbb{P}\left[\exists u\geq 0,\left((\hat{B}_{1+u}-\hat{B}_1)-\lambda u\right)>(\lambda-m)\right]\\
&=&\mathbb{P}\left[\sup_{u\geq 0}(B_u-\lambda u)>(\lambda-m)\right]\\[0.1cm]
&=&\mathbb{P}\left[\sup_{u\geq 0}\exp{\left(2\lambda(B_u-\lambda u)\right)}>\exp{\left(2\lambda(\lambda-m)\right)}\right]\\[0.1cm]
&=&\mathbb{P}\left[\frac{1}{\mathbf{U}}>\exp{\left(2\lambda(\lambda-m)\right)}\right]\\
&=&\mathbb{P}\left[\mathbf{U}<\exp{\left(-2\lambda(\lambda-m)\right)}\right]\\
&=&\exp{\left(-2\lambda(\lambda-m)\right)}.
\end{eqnarray*}
We note that since we derive formula (\ref{united}) from Doob's
maximal identity, we can recover the joint law of $(S_1,B_1)$ from
this identity; traditionally, this joint law is obtained from the reflection principle.
\end{enumerate}

\begin{exo}
\begin{enumerate}
\item[a)] Prove (\ref{point6.4}) with the help of the absolute continuity
  relationship between $\mathbb{P}_{0\to m}^{(u)}$ and $\mathbb{P}_0$
  on $\mathcal{F}_t$, for any $t<u $.
\item[b)] Compute the law of $T_\lambda$ under $P_{0\to m}^{(u)}$.\\
The answer is :
\[P_{0\to m}^{(u)}\left(T_\lambda\in dt\right)=dt\frac{\lambda}{t}p_t(0,\lambda)\frac{p_{u-t}(\lambda,m)}{p_{u}(0,m)}.\]
\end{enumerate}
\end{exo}

\begin{exo}
Compute :
\[\mathbb{P}\left(\sup_{t\leq
    1}\left(\frac{B_t}{a+bt}\right)>\lambda|B_1=m\right).\]
The previous method yields :
\[
\mathbb{P}\left(\sup_{t\leq
    1}\left(\frac{B_t}{a+bt}\right)>\lambda|B_1=m\right)=
\exp{\left(-2\lambda a(\lambda (a+b)-m)\right)},
\]
for $\lambda>\frac{m}{a+b}$, $\lambda>0$.
\end{exo}

\subsection{Towards an integral representation of the positive $PFH$ functions}
So far, we have discovered two families of $PFH$ functions, i.e :
$h^{(l,\,\nu)}$ on one hand and $h^{(a)}$ on the other hand. It would
be nice to exhibit a unifying formula for all these $PFH$ functions.

In fact, we are interested in the following question :
is there an integral representation of all positive $PFH$ functions?
We were motivated to raise this question from the following classical
result, due to D. Widder \cite{widder}, see also J.L. Doob
\cite{doob}, concerning positive space-time harmonic functions (of
Brownian motion) :
$h\geq 0$ is space-time harmonic, if, by definition, $\{h(B_s,s),
s\geq 0\}$ is a Brownian martingale.
Thus, there exists a positive finite measure $\mu(d\lambda)$
such that :
\begin{equation}\label{6.5.160}
h(x,s)=\int_{-\infty}^\infty \,d\mu(\lambda)\exp{\left(\lambda x-\frac{\lambda^2s}{2}\right)}.
\end{equation}
\underline{Comment on (\ref{6.5.160}):}\\
In chapter I of \cite{zurich}, a probabilistic proof of Widder's result is obtained with the help of the non canonical Brownian motion :
\[B'_t=B_t-\int_0^t\frac{ds}{s}\,B_s,\quad t \geq 0.\]
It would be most interesting to obtain a representation of the $PFH$ functions by constructing a two-parameter process, with a past-future filtration $\mathcal{F}'_{s,t}$ such that :
\[\mathcal{F}_{s,t}=\mathcal{F}'_{s,t}\vee \sigma(B_s,B_t),\]
and $\mathcal{F}'_{s,t}$ is independent from the pair $(B_s,B_t)$.Such a
filtration is exhibited in the next exercise :
\begin{exo}
The past-future Bridge filtration.
\begin{enumerate}
\item Let us consider the following two-parameter filtration :
\[\mathcal{F}_{s,t}'=\sigma\{B_u-\frac{u}{s}B_s, u\leq s; B_{t+h}-B_t,h\geq
0\}.\]
Note that the $\sigma$-field $\mathcal{F}_{s,t}$ satisfies :
\[\mathcal{F}_{s,t}=\mathcal{F}_{s,t}'\vee\sigma(B_s,B_t),\]
and that $\mathcal{F}_{s,t}'$ is independent from $\sigma(B_s,B_t)$, in
fact from $\mathcal{F}_{[s,t]}=\sigma\{B_u, s\leq u\leq t\}$.
\item Prove that $\left(\mathcal{F}_{s,t}'\right)_{0\leq s<t<\infty}$ is
  the natural two-parameter filtration associated with :
\[X_{s,t}=\left(B_s-\int_0^s
  \frac{du}{u}\,B_u\right)+\left(\frac{1}{t}B_t-\int_t^\infty
  \frac{dv}{v^2}\,B_v\right),\]
that is : 
\[\mathcal{F}_{s,t}'=\Xi_{s,t}\equiv\sigma\{X_{u,v};u\leq s, v\geq t\}.\]
Hint : simply show that :
\[\mathbb{E}[X_{s,t}B_s]=\mathbb{E}[X_{s,t}B_t]=0,\]
and proceed from there.
\item Note that :
\[X_{s,t}=X_s^{(-)}+X_t^{(+)},\]
where : $X_s^{(-)}=\beta(B)_s$; $X_t^{(+)}=\beta(\hat{B})_{1/t}$, and :
$\beta(B)_s=B_s-\int_O^s\frac{du}{u}B_u$ is the Brownian motion introduced
in Chapter I of Yor, Zürich \cite{zurich}.
Furthermore : $\hat{B}_u=uB_{1/u}$. 
\end{enumerate}
\end{exo}
\begin{question}
It may also be of interest to consider the family of
interval-$\sigma$-field : $\mathcal{F}_{[s,t]}=\sigma\{B_u, s\leq u\leq t\}$,
and to look for the ``generation'' of $\mathcal{F}_{[s,t]}'$, a sub
$\sigma$-fields of $\mathcal{F}_{[s,t]}$ :
\[\mathcal{F}_{[s,t]}'=\sigma\{\frac{B_u-B_v}{u-v}-\frac{B_t-B_s}{t-s};s\leq
u\leq v\leq t\}.\]
Can one give a simpler description of $\mathcal{F}_{[s,t]}'$? Note that, if
$I_1, \dots, I_k$ are disjoint intervals then $\mathcal{F}_{I_1}', \dots,
\mathcal{F}_{I_k}'$ are independent
\end{question}

Independently from the previous exercise, we have obtained the following partial answers presented in the next two propositions.
\begin{proposition}\label{famille1}
Let $h(s,t;x,y)=\phi(\frac{1}{t-s};\frac{y-x}{t-s})$ for a function
$\phi(u,z)$.
Then, $h$ is $PFH$ if and only if $\phi$ is space-time harmonic, i.e :
\begin{equation}
\phi'_u+\frac{1}{2}\phi''_{z^2}=0.
\end{equation}
\end{proposition}
\begin{proof}
The $LHS$ of the first equation in $(-)$ writes, in this case :
(with : $u=1/(t-s)$; $z=(y-x)/(t-s)$)
\begin{equation}
\frac{1}{(t-s)^2}\phi'_u(u,z)-\frac{y-x}{(t-s)^2}\phi'_z(u,z)+\frac{y-x}{(t-s)^2}\phi'_z(u,z)+\frac{1}{2(t-s)^2}\phi''_{z^2}.
\end{equation}
Hence, the result.
\end{proof}
\begin{proposition}\label{famille2}
Let $h(s,t;x,y)=F(\frac{x}{\sqrt{t-s}},\frac{y}{\sqrt{t-s}})$ for a
function $F$. Then, $h$ is $PFH$ if and only if :
\begin{equation}\label{systou}
\begin{cases}
aF'_a+(2a-b)F'_b+F''_{b^2}=0\\
(2b-a)F'_a+bF'_b+F''_{a^2}=0
\end{cases}
\end{equation}
\end{proposition}
\begin{proof}
\[h'_s=\frac{1}{2}\frac{x}{(t-s)^{3/2}}F'_a+\frac{1}{2}\frac{y}{(t-s)^{3/2}}F'_b;\]
\[h'_x=\frac{1}{\sqrt{t-s}}F'_a;\]
\[h''_{x^2}=\frac{1}{t-s}F''_{a^2}.\]
We must have :
\begin{eqnarray*}
&\frac{1}{2}\frac{1}{t-s}\left[\frac{x}{\sqrt{t-s}}F'_a+\frac{y}{\sqrt{t-s}}\right]
&+\frac{1}{t-s}\frac{y-x}{\sqrt{t-s}}F'_a+\frac{1}{2(t-s)}F''_{a^2}=0\\
\end{eqnarray*}
Hence :
\[\frac{1}{2}\left(aF'_a+bF'_b\right)+(b-a)F'_a+\frac{1}{2}F''_{a^2}=0.\]
We leave the end of the proof to the reader.
\end{proof}
Note that $F$, solution of (\ref{systou}), is harmonic for the
2-dimensional Ornstein-Uhlenbeck process $(X,Y)$ defined by :
\[X_t=\beta_t+\int_0^tY_s\,ds; Y_t=\gamma_t+\int_0^tX_s\,ds,\]
where $\beta$ and $\gamma$ is a 2-dimensional Brownian motion.

\textbf{A question :}\\
We dispose of the equations $(-)$ and $(+)$. We would
like to find a more general family of $PFH$ functions, than the ones
we obtained in Proposition \ref{famille1} and Proposition
\ref{famille2}.
More precisely, we consider the functions defined by :
\[h(s,t;x,y)=\exp{\left(a(s,t)xy+b(s,t)x+c(s,t)y+d(s,t)\right)},\]
associated to the four time dependent functions $a$, $b$, $c$, $d$.
\begin{enumerate}
\item[$\bullet$]$h$ satisfies formula $(-)$ if and only if : (after
computation of $h'_s$, $h'_x$, $h'_{x^2}$ and identification of the terms
in $xy$, $y^2$, $x$, $y$ and the constant parts)
\begin{equation}
\begin{cases}
(1)\quad \frac{\partial a}{\partial s}=\frac{1}{t-s}a(s,t)\\[0.1cm]
(2)\quad a(s,t)\left(\frac{1}{t-s}+\frac{1}{2}a(s,t)\right)=0\\[0.1cm]
(3)\quad \frac{\partial b}{\partial s}=\frac{1}{t-s}b(s,t)=0\\[0.1cm]
(4)\quad \frac{\partial c}{\partial s}+b(s,t)\left(\frac{1}{t-s}+a(s,t)\right)=0\\[0.1cm]
(5)\quad \frac{\partial d}{\partial s}+\frac{1}{2}b(s,t)^2=0
\end{cases}
\end{equation}

$(2)$ gives :
\[a(s,t)=0,\quad \mathrm{or}\quad a(s,t)=-\frac{2}{t-s}.\]
$(3)$ gives :
\[b(s,t)=\frac{\beta(t)}{t-s}.\]
Let us take $a(s,t)=-\frac{2}{t-s}$ (leaving the case $a(s,t)=0$ aside for the moment).
Then
$(4)$ gives :
\[c(s,t)=\frac{\beta(t)}{t-s}+\gamma(t).\]
$(5)$ gives :
\[d(s,t)=-\frac{1}{2}\frac{\beta(t)^2}{t-s}+\delta(t).\]

\item[$\bullet$]$h$ satisfies formula $(+)$ if and only if : (after
computation of $h'_t$, $h'_y$, $h'_{y^2}$ and identification of the terms
in $xy$, $y^2$, $x$, $y$ and the constant parts)
\begin{equation}
\begin{cases}
(1')\quad -\frac{\partial a}{\partial t}=\frac{1}{t-s}a(s,t)\\[0.1cm]
(2')\quad a(s,t)\left(\frac{1}{t-s}+\frac{1}{2}a(s,t)\right)=0\\[0.1cm]
(3')\quad -\frac{\partial b}{\partial t}+a(s,t)c(s,t)+\frac{c(s,t)}{t-s}=0\\[0.1cm]
(4')\quad \frac{\partial c}{\partial t}\left(\frac{1}{t-s}\right)+c(s,t)=0\\[0.1cm]
(5')\quad \frac{\partial d}{\partial t}+c(s,t)^2=0
\end{cases}
\end{equation}
Again $(2')$ gives :
\[a(s,t)=0,\quad \mathrm{or}\quad a(s,t)=-\frac{2}{t-s}.\]
$(4')$ gives :
\[c(s,t)=\hat{\gamma}(s)(t-s).\]
$(3')$ gives :
\[b(s,t)=-\hat{\gamma}(s)t+\hat{\beta}(s).\]
$(5')$ gives :
\[d(s,t)=\hat{\gamma}(s)^2\frac{(t-s)^3}{3}+\hat{\delta}(s).\]
\end{enumerate}
We now have to compare the different expressions we have obtained for
$a$, $b$, $c$, $d$ and determine 6 functions of one variable, i.e :
$\beta$, $\gamma$, $\delta$  on one hand and $\hat{\beta}$, $\hat{\gamma}$, $\hat{\delta}$ on the other hand.

\subsubsection{A first attempt to represent $PFH\geq0$ functions :}
Let $h(s,t;B_s,B_t)$ be an $\mathcal{F}_{s,t}$-martingale. Then, for
fixed $t$,
\[s\to h(s,t;B_s,B_t)\quad(s<t)\: \mathrm{is}\:\mathrm{an}\:
\mathcal{F}_s^{(t)} \mathrm{martingale}.\]
Hence : for fixed $(t,y)$, $h(s,t;x,y)$ is space-time harmonic for $\mathbb{P}_{0\to
  y}^{(t)}$, that is :
$h(s,t;B_s,y)p_{t-s}(B_s,y)$ is a martingale, (but only until
$t$) for Brownian motion, i.e (with obvious notations) :
\[H_{t,y}(s,x)=h(s,t;x,y)p_{t-s}(x,y) \mathrm{is\:space-time\:harmonic\:in} (s,x).\]
We shall note : $(s,x)$-space-time harmonic.
Hence,
\begin{equation}\label{-}
h(s,t;x,y)=\frac{1}{p_{t-s}(x,y)}H_{t,y}(s,x).
\end{equation}

Moreover, thanks\footnote{Here, our treatment is not totally rigorous since we assume that an $(s,x)$-space-time harmonic function, for $s<t$, may be extended to an $(s,x)$-space-time harmonic, for all $s>0$. Can one characterize such $t$-space-time harmonic functions? Our second attempt (see 6.5.2) does not suffer from this abuse.} to Proposition \ref{timeinver}, we have also :
\begin{equation}\label{+}
h(\frac{1}{t},\frac{1}{s};\frac{y}{t},\frac{x}{s})=\frac{1}{p_{t-s}(x,y)}K_{t,y}(s,x),
\end{equation}
where $K_{t,y}$ is a $(s,x)$-space-time harmonic function.
Let us rewrite (\ref{+}) with $u=\frac{1}{t}$, $v=\frac{1}{s}$,
$x'=\frac{y}{t}$, $y'=\frac{x}{s}$ :
\begin{equation}\label{+bis}
\begin{split}
h(u,v;x',y')&=\frac{1}{p_{\frac{1}{u}-\frac{1}{v}}(sy',tx')}K_{\frac{1}{u},tx'}(\frac{1}{v},sy')\\
&=\frac{1}{p_{\frac{1}{u}-\frac{1}{v}}(\frac{y'}{v},\frac{x'}{u})}K_{\frac{1}{u},\frac{x'}{u}}(\frac{1}{v},\frac{y'}{v}).
\end{split}
\end{equation}

Thus, $H$ and $K$ need to satisfy : 
\begin{equation}\label{A}
h(s,t;x,y)=\frac{1}{p_{t-s}(x,y)}H_{t,y}(s,x)=\frac{1}{p_{\frac{1}{s}-\frac{1}{t}}(\frac{y}{t},\frac{x}{s})}K_{\frac{1}{s},\frac{x}{s}}(\frac{1}{t},\frac{y}{t}).
\end{equation}

After some elementary computation, we have :
\begin{equation}
\frac{1}{p_{\frac{1}{s}-\frac{1}{t}}(\frac{y}{t},\frac{x}{s})}=\frac{\sqrt{st}}{\sqrt{2\pi(t-s)}}\exp{\left(-\frac{1}{2}\frac{(sy-xt)^2}{st(t-s)}\right)}.
\end{equation}
Hence, (\ref{A}) becomes :
\begin{equation}\label{fin}
\begin{split}
H_{t,y}(s,x)&=\frac{1}{\sqrt{st}}\exp{\left(\frac{(sy-xt)^2}{st(t-s)}-\frac{(y-x)^2}{t-s}\right)}K_{\frac{1}{s},\frac{x}{s}}(\frac{1}{t},\frac{y}{t})\\
&=\frac{1}{\sqrt{st}}\exp{\left(-\frac{1}{2}\left(\frac{y^2}{t}-\frac{x^2}{s}\right)\right)}K_{\frac{1}{s},\frac{x}{s}}(\frac{1}{t},\frac{y}{t}).
\end{split}
\end{equation}

If we assume that :
\[K_{\frac{1}{s},\frac{x}{s}}(\frac{1}{t},\frac{y}{t})=\int_{-\infty}^\infty\,d\eta\,\mathcal{K}_{\frac{1}{s},\frac{x}{s}}(\eta)\exp{\left(\eta\frac{y}{t}-\frac{\eta^2}{2t}\right)},\] 
i.e : the representing measure of $K_{a,z}(s,x)$is absolutely
continuous, it ``is likely'' that :
\[
\frac{1}{\sqrt{s}}\exp{\left(-\frac{x^2}{2s}\right)}\mathcal{K}_{\frac{1}{s},\frac{x}{s}}(\eta),
\]is $(s,x)$-space-time harmonic, i.e :
\begin{equation}
\frac{1}{\sqrt{s}}\exp{\left(-\frac{x^2}{2s}\right)}\mathcal{K}_{\frac{1}{s},\frac{x}{s}}(\eta)=\int\,\tau_{\eta}(d\lambda)\exp{\left(\lambda
  x-\frac{\lambda^2 s}{2}\right)}.
\end{equation}
Using again (\ref{A}), and taking $t=\frac{1}{u}$,
$\frac{y}{t}=uy=z$, we have :
\begin{equation}
K_{\frac{1}{s},\frac{x}{s}}(u,z)=\sqrt{\frac{u}{s}}\exp{\left(\frac{1}{2}\left(\frac{x^2}{s}+z^2t\right)\right)}H_{\frac{1}{u},\frac{z}{u}}(s,x)
\end{equation}
If we assume that :
\[H_{t,y}(s,x)=\int_{-\infty}^\infty\,d\lambda\,\mathcal{H}_{t,y}(\lambda)\exp{\left(\lambda
    x-\frac{\lambda^2s}{2}\right)},\] 
it ``is likely'' that :
\[
\frac{1}{\sqrt{u}}\exp{\left(\frac{z^2}{u}\right)}\mathcal{H}_{\frac{1}{u},\frac{z}{u}}(\lambda),
\]is $(u,z)$-space-time harmonic, i.e :
\begin{equation}
\frac{1}{\sqrt{u}}\exp{\left(\frac{z^2}{u}\right)}\mathcal{H}_{\frac{1}{u},\frac{z}{u}}(\lambda)
=\int\,\theta_{\lambda}(d\rho)\exp{\left(\rho
  z-\frac{\rho^2 u}{2}\right)}.
\end{equation}
It now remains to find a correspondence between $\tau_{\eta}$ and $\theta_{\lambda}$.......

\begin{exo}
Show from the above arguments that :
\[e_s^{(t)}\equiv \exp{\left(-\frac{2B_sB_t}{t-s}\right)}\:\mathrm{is\:a\:}(\mathcal{F}_{s,t}, s<t)\:\mathrm{martingale}.\]
More generally, if $(B_u)$  is a $n$-dimensional Brownian motion, show that :
\[\exp{\left(-\frac{2B_s\bullet B_t}{t-s}\right)}\:\mathrm{is\:a\:PF\:martingale}.\]
Solution : We need to show in particular that :
\[\exp{\left(-\frac{2B_sm}{t-s}\right)}p_{t-s}(B_s,m)\:\mathrm{is\:a\:}(\mathcal{F}_{s},s<t)\:\mathrm{martingale}.\]
In fact, this expression is : $p_{t-s}(B_s,-m)$, which is of course a $(\mathcal{F}_s, s<t)$ martingale!!
\end{exo}
\begin{exo}
( : T. Fujita, M. Yor \cite{fujita}). 
\renewcommand{\theenumi}{\roman{enumi}.}
\renewcommand{\labelenumi}{\theenumi}
\begin{enumerate}
\item Assume $(X_t)$ is a symmetric Lévy process, with
semigroup density : $p_t(x,y)\equiv q_t(|x-y|)$, for $q_t$ a function
on $\mathbb{R}^+$. Prove that, for any $a\in\mathbb{R}$, 
\[\frac{p_{t-s}(X_s+a,-(X_t+a))}{p_{t-s}(X_s,X_t)},\]
is a $\mathcal{F}_{s,t}$ martingale.
\item Application : for the Cauchy process $X_t=C_t$, one gets :
\[\frac{(t-s)^2+(C_t-C_s)^2}{(t-s)^2+(C_t+C_s+2a)^2}\]
is a $(\mathcal{F}_{s,t})$ martingale.
\end{enumerate}
Question : does there exist an adequate extension of this result to
non-symmetric Lévy processes?
\end{exo}
\begin{exo}
Define $R_u=|B_u|$, the n-dimensional Bessel process associated with $B$ a Brownian
motion in $\mathbb{R}^n$; compute the projection of $e_s^{(t)}$ on $\mathcal{R}_{s,t}$.
\end{exo}
\begin{exo}
Show that : $e_s^{(t)}\to 0$ when $s\to t$.\\
Identify the law of $\mathcal{G}_1^{(t)}=\sup\{s<t, e_s^{(t)}=1\}$, and more generally the law of $\mathcal{G}_K^{(t)}$
\end{exo}

\subsubsection{A second attempt to represent $PFH\geq0$ functions :}
Recall again that : $h(s,t;x,y)$ is $PFH$ if and only if :
\renewcommand{\theenumi}{\roman{enumi}.}
\renewcommand{\labelenumi}{\theenumi}
\begin{enumerate}
\item $\forall (t,y),\quad (s,x)\to h(s,t;x,y)$ is harmonic under $\mathbb{P}_{0\to y}^{(t)}$;
\item $h(s,t;x,y)\equiv k(\frac{1}{t},\frac{1}{s};\frac{y}{t},\frac{x}{s})$, where $k$ is $PFH$; and we shall apply (i) to $k$, for fixed $(s,x)$.
\end{enumerate}
Based on this, we can state the following :
\begin{proposition}\label{chéplus}
$h$ is $PFH\geq 0$ if and only if :
\begin{enumerate}
\item[a)]$\forall (t,y)$, there exists a space-time harmonic function for Brownian motion : $K_{(t,y)}^{(+)}$ such that :
\begin{equation}
h(s,t;x,y)=K_{(t,y)}^{(+)}\left(\frac{s}{t-s};\frac{xt-ys}{\sqrt{t}(t-s)}\right)
\end{equation}
\item[b)]$\forall (s,x)$, there exists a space-time harmonic function for Brownian motion : $K_{(\frac{1}{s},\frac{x}{s})}^{(-)}$ such that :
\begin{equation}
h(s,t;x,y)=K_{(\frac{1}{s},\frac{x}{s})}^{(-)}\left(\frac{s}{t-s};\frac{\sqrt{s}(y-x)}{t-s}\right)
\end{equation}
\end{enumerate}
\end{proposition}
To help the reader with these slightly complicated formulae, let us compute $K^{(+)}$ and $K^{(-)}$ for : 
\[h(s,t;x,y)=\exp{\left(-\frac{2xy}{t-s}\right)},\]in the following exercise :
\begin{exo}
\begin{enumerate}
\item[a)]Fix $t$ and $m$; find $K^{(+)}$ a space-time harmonic function such that :
\[\exp{\left(-\frac{2xm}{t-s}\right)}=K^{(+)}\left(\frac{s}{t-s};\frac{xt-ms}{\sqrt{t}(t-s)}\right).\] 
Answer :
\[K^{(+)}(u,\xi)=\exp{\left(-2m'\xi-2m'^2u\right)},\]
where $m'=\frac{m}{\sqrt{t}}$.
(Note the remarkable fact that $K^{(+)}$ depends only on $m'$).
\item[b)] Fix $s$ and $x$; find $K^{(-)}$ a space-time harmonic function such that :
\[\exp{\left(-\frac{2xy}{t-s}\right)}=K^{(-)}\left(\frac{s}{t-s};\frac{\sqrt{s}(y-x)}{t-s}\right).\] 
Answer :
\[K^{(-)}(u,\xi)=\exp{\left(-2x'(\xi+x'u)\right)},\]
where : $x'=x/\sqrt{s}$.\\
(Note the remarkable fact that $K^{(-)}$ depends only on $x'$).
\end{enumerate}
\end{exo}

\begin{proof} (of Proposition \ref{chéplus})\\
As explained just before stating that Proposition, the proof consists in expressing the generic $\geq 0$, space-time harmonic function $h_m^{(t)}(s,x)$ for the Brownian bridge $BB_{0\to m}^{(t)}$ in terms of a space-time harmonic function for Brownian motion.
The result is :
\begin{equation}\label{c}
h_m^{(t)}(s;x)=K\left(\frac{s}{t-s};\frac{xt-ms}{\sqrt{t}(t-s)}\right).
\end{equation}
we denote by $\mathcal{H}_m^{(t)}$ the set of the $h_m^{(t)}$ functions, space-time harmonic functions for the Brownian bridge $BB_{0\to m}^{(t)}$, and by $\mathcal{H}$ the set of the space-time harmonic functions for Brownian motion.
\begin{enumerate}
\item[1.] Clearly, by scaling,
\begin{equation}\label{a}
h_m^{(t)}(s;x)=h_{m/\sqrt{t}}^{(1)}\left(\frac{s}{t};\frac{x}{\sqrt{t}}\right).
\end{equation}
\item[2.] Let $h\in\mathcal{H}_m^{(1)}$; then, we shall show that there exists $K\in\mathcal{H}$ such that :
\begin{equation}\label{b}
h\left(\frac{1}{\tau+1};\frac{m+y}{\tau+1}\right)=K\left(\frac{1}{\tau};\frac{y}{\tau}\right).
\end{equation}
Then, we obtain (\ref{c}) from (\ref{a}) and (\ref{b}).
\item[3.] It now remains to prove (\ref{b}).\\
In order to state properties in Brownian terms, and not in terms of Brownian bridges, we use time inversion (in a first instance) :
\[B'_v=vB_{1/v}.\]
Then, the martingale property for $h$ writes : $(s<s'<1)$
\begin{eqnarray*}
\mathbb{E}\left[F\left(uB'_{1/u},\frac{1}{u}\geq\frac{1}{s'}\right)h\left(s;sB'_{1/s}\right)|B'_1=m\right]=\\[0.1cm]
\mathbb{E}\left[F\left(uB'_{1/u},\frac{1}{u}\geq\frac{1}{s'}\right)h\left(s';s'B'_{1/s'}\right)|B'_1=m\right].
\end{eqnarray*}
We use : $\sigma=1/s$; $\sigma'=1/s'$, then :
\begin{eqnarray*}
\mathbb{E}\left[F\left(B'_v, v\geq
    \sigma'\right)h\left(\frac{1}{\sigma};\frac{B'_{\sigma}}{\sigma}\right)|B'_1=m\right]=\\[0.1cm]
\mathbb{E}\left[F\left(B'_v, v\geq \sigma'\right)h\left(\frac{1}{\sigma'};\frac{B'_{\sigma'}}{\sigma'}\right)|B'_1=m\right].
\end{eqnarray*}
We now shift time by $1$ : $\sigma=\tau+1$; $\sigma'=\tau'+1$; and we introduce : $\tilde{B}_u=B'_{(u+1)}-B'_1$, a new Brownian motion. Then :
\begin{eqnarray*}
\mathbb{E}\left[F\left(\tilde{B}_u, u\geq
    \tau'\right)h\left(\frac{1}{\tau+1};\frac{m+\tilde{B}_{\tau}}{\tau+1}\right)\right]=\\[0.1cm]
\mathbb{E}\left[F\left(\tilde{B}_u, u\geq \tau'\right)h\left(\frac{1}{\tau'+1};\frac{m+\tilde{B}_{\tau'}}{\tau'+1}\right)\right].
\end{eqnarray*}
It is advantageous to introduce the notation :
\[H(\tau,x)=h\left(\frac{1}{\tau+1};\frac{m+x}{\tau+1}\right).\]
The previous equality writes :
\begin{eqnarray*}
\mathbb{E}\left[F\left(\tilde{B}_u, u\geq \tau'\right)H\left(\tau;\tilde{B}_{\tau}\right)\right]=\mathbb{E}\left[F\left(\tilde{B}_u, u\geq \tau'\right)H\left(\tau';\tilde{B}_{\tau'}\right)\right].
\end{eqnarray*}
Finally, we shall bring time ``back in order'' by introducing $\beta$ such that : $\tilde{B}_u=u\beta_{1/u}$, (time inversion : second instance), and : $\frac{1}{\tau'}=\theta'<\frac{1}{\tau}=\theta$. Hence :
\begin{equation}
\begin{split}
\mathbb{E}\left[F\left(\beta_{1/k}, \frac{1}{k}\leq
    \theta'\right)H\left(\frac{1}{\theta};\frac{1}{\theta}\beta_\theta\right)\right]&=\\
&\mathbb{E}\left[F\left(\beta_{1/k},
    \frac{1}{k}\leq
    \theta'\right)H\left(\frac{1}{\theta'};\frac{1}{\theta'}\beta_{\theta'}\right)\right].
\end{split}
\end{equation}
Consequently, $K(\theta;x)=H\left(\frac{1}{\theta};\frac{x}{\theta}\right)$ belongs to $\mathcal{H}$. We have obtained (\ref{b}).
\end{enumerate}
\end{proof}

\subsection{An explicit expression for the law of  $g_x^{(\nu)}(t)$}
In Note $2$, we discussed the Black-Scholes formula in relation with the law
of $\mathcal{G}_K$.
Similarly, here in our finite horizon framework, we would like to give an explicit
expression for the law of :\[g_x^{(\nu)}(t)=\sup\{s\leq t,
B_s^{(\nu)}=x\}.\]
Already, in the case $\nu=0$, the following result has been obtained
(see \cite{yor} and alsoo \cite{jeanblanc} Proposition 4.3.3.3) :
\begin{equation}\label{etoile6.6}
\mathbb{P}\left(g_x^{(0)}(t)\in
du\right)=\frac{du}{\pi\sqrt{u(t-u)}}\exp{\left(-\frac{x^2}{2u}\right)}\quad (0<u<t).
\end{equation}
Note that this is a sub-probability, since
$\frac{du}{\pi\sqrt{u(t-u)}}\quad (0<u<t)$ is the arcsin distribution;
indeed,
\begin{eqnarray*}
\mathbb{P}\left(g_x^{(0)}(t)=0\right)&=&\mathbb{P}\left(T_x\geq
  t\right)\\
&=&\mathbb{P}\left(|\mathbf{N}|\leq \frac{|x|}{\sqrt{t}}\right).
\end{eqnarray*}
We are now interested in the general case : $\nu\ne 0$; we shall
compute, for $0\leq s\leq t$ :
\[\gamma_{x,\,t}^{(\nu)}(s)\overset{\underset{\mathrm{def}}{}}{=}\mathbb{P}(0<g_x^{(\nu)}(t)\leq
s).\]
\begin{equation}
\gamma_{x,\,t}^{(\nu)}(s)=\mathbb{E}\left[1_{\{0<g_x^{(0)}(t)\leq s\}}\exp{\left(\nu B_t-\frac{\nu^2t}{2}\right)}\right].
\end{equation}
On $\{0<g_x^{(0)}(t)\leq s\}$, we have : $T_x \leq s$, and introducing
:
\[\tilde{B}_u=B_{T_x+u}-x,\]
a new Brownian motion, independent from $\mathcal{F}_{T_x}$, we have :
\[g_x^{(0)}(t)=T_x+\tilde{g}(t-T_x),\]
where $\tilde{g}(h)=\sup\{s<h, \tilde{B}_s=0\}$.
We have :
\begin{eqnarray*}
\gamma_{x,\,t}^{(\nu)}(s)&=&\mathbb{E}\left[1_{\{T_x\leq
    s\}}1_{\{\tilde{g}(t-T_x)\leq
    s-T_x\}}\exp{\left(\nu(x+\tilde{B}_{t-T_x})\frac{\nu^2t}{2}\right)}\right]\\
&=&\exp{\left(\nu
    x-\frac{\nu^2t}{2}\right)}\\
&\:&\mathbb{E}\left[1_{\{T_x+\tilde{g}(t-T_x)\leq
    s\}}\exp{\left(\nu\sqrt{(t-T_x)-\tilde{g}(t-T_x)}\tilde{m}_1\right)}\right], 
\end{eqnarray*}
where we have used the factorisation :
\[\tilde{B}_{t-T_x}=\sqrt{(t-T_x)-\tilde{g}(t-T_x)}\tilde{m}_1,\]
with $\tilde{m}_1=\epsilon m_1$ independent from
$(T_x,\tilde{g}(t-T_x))$, $\epsilon$ is Bernoulli,
$m_1\overset{\underset{\mathrm{law}}{}}{=}\sqrt{2\mathbf{e}}$.
We introduce the function :
\begin{equation}
\phi(\lambda)=\mathbb{E}[\exp{(\lambda\tilde{m}_1)}]
=\mathbb{E}[\cosh{(\lambda m_1)}]
=\int_0^\infty \,dt\,e^{-t}\cosh{(\lambda\sqrt{2t})}.
\end{equation}
An integration by parts shows that :
\begin{equation}
\phi(\lambda)=1+e^{\lambda^2/2}|\lambda|\sqrt{2\pi}\mathbb{P}\left(|\mathbf{N}|\leq
|\lambda|\right).
\end{equation}
See \cite{azyor3} for similar computations.
Consequently, we obtain :
\begin{equation}
\gamma_{x,\,t}^{(\nu)}(s)=\mathbb{P}\left({1_{\{0<g_x(t)\leq s\}}e^{\nu
    x-\frac{\nu^2t}{2}}}\phi(\nu\sqrt{t-g_x(t)})\right),
\end{equation}
which from (\ref{etoile6.6}), gives :
\begin{equation}
\mathbb{P}(g_x^{(\nu)}(t)\in du)=\frac{du}{\pi\sqrt{u(t-u)}}\exp{\left(-\frac{x^2}{2u}\right)}\phi(\nu\sqrt{t-u}).
\end{equation}
Very close computations may be found in Chapter 4, subsection 4.3.9 in \cite{jeanblanc}.
\subsection{On the time spent below a level by Brownian motion with
  drift}
Consider :
\[A_x^{(\nu)}(t)=\int_0^t\,ds\,1_{\{B_s^{(\nu)}\leq x\}},\quad x>0.\]
Thus, it is shown in \cite{yor}, that, for $\nu=0$ :
\begin{equation}
A_x^{(0)}(t)\overset{\underset{\mathrm{law}}{}}{=}g_x^{(0)}(t).
\end{equation}
\begin{question}
We may ask whether this identity in law also holds with $(B^{(\nu)})$
instead of $(B)$.
\end{question}
\begin{question}
Another question consists in studying :
\[\mathbb{P}\left(A_x^{(\nu)}(t)\leq s|\mathcal{F}_s\right).\]
(Maybe already studied by Dassios \cite{dass}, and Embrechts, Rogers
and Yor \cite{embr})
\end{question}

\subsection{Another class of past-future martingales}
(To appear in second edition of Chaumont-Yor, \cite{chaumont}, Chapter 7).
We are interested in constructing $(\mathcal{F}_{s,t})$ martingales
from the simple recipe :\\
if $F\in L^1(\mathbb{W},\mathcal{F}_\infty)$, then :
\[M_{s,t}(F)\overset{\underset{\mathrm{def}}{}}{=}\mathbb{E}\left[F|\mathcal{F}_{s,t}\right]\:\mathrm{is\:a}\:
(\mathcal{F}_{s,t})\mathrm{-martingale}.\]
For a number of functionals $F$, we are able to compute $M_{s,t}(F)$.
\begin{ex}
\begin{equation}\label{boom}
F=\exp{\left(\int_0^\infty\,f(u)dB_u-\frac{1}{2}\int_0^{\infty}\,f^2(u)du \right)},
\end{equation}
where $f\in L^2\left([0,\infty),du\right)$.
Then,
\begin{eqnarray*}
\mathbb{E}\left[F|\mathcal{F}_{s,t}\right]&=&\exp{\left(\int_0^s\,f(u)dB_u+\int_t^\infty\,f(u)dB_u-\frac{1}{2}\int_0^{\infty}\,f^2(u)du\right)}\\
&\times&\mathbb{E}\left[\exp{\left(\int_s^t\,f(u)dB_u\right)}|\mathcal{F}_{s,t}\right].
\end{eqnarray*}
We know that :
\begin{eqnarray*}
\mathbb{E}\left[\int_s^t\,f(u)dB_u|\mathcal{F}_{s,t}\right]&=&\frac{1}{t-s}\left(\int_s^t\,f(u)du\right)\left(B_t-B_s\right)\\
&=&\left(\int_s^t\,f(u)du\right)B_{[s,t]}.
\end{eqnarray*}
Hence :
\begin{eqnarray*}
\mathbb{E}\left[\exp{\left(\int_s^t\,f(u)dB_u\right)}|\mathcal{F}_{s,t}\right]&=&\exp{\left(\left(\int_s^t\,f(u)du\right)B_{[s,t]}\right)}\\[0.1cm]
&&\exp{\frac{1}{2}\left(\int_s^t\,f^2(du)du-\frac{1}{t-s}\left(\int_s^t\,f(u)du\right)^2\right)}.
\end{eqnarray*}
Finally, we have obtained :
\begin{equation}
\mathbb{E}\left[F|\mathcal{F}_{s,t}\right]=F_s^{(1)}F_{s,t}^{(2)}F_t^{(3)},
\end{equation}
where :
\begin{equation}
\begin{cases}
F_s^{(1)}=\exp{\left(\int_0^s\,f(u)dB_u-\frac{1}{2}\int_0^s\,f^2(u)du\right)}\\[0.15cm]
F_{s,t}^{(2)}=\exp{\left(\left(\int_s^t\,f(u)du\right)B_{[s,t]}-\frac{1}{2(t-s)}\left(\int_s^t\,f(u)du\right)^2\right)}\\[0.15cm]
F_t^{(3)}=\exp{\left(\int_t^\infty\,f(u)dB_u-\frac{1}{2}\int_t^\infty\,f^2(u)du\right)}.
\end{cases}
\end{equation}
\end{ex}
\textbf{Wishful thinking :} Recall that the computation of  $\mathbb{E}[F|\mathcal{F}_t]$, for $F$ in (\ref{boom}), and the fact that these functionals are total in $L^1(\mathbb{W},\mathcal{F}_{\infty})$ easily leads to the representation result of Brownian martingales as stochastic integrals with respect to Brownian motion. Could there be an analogous result for the $(\mathcal{F}_{s,t})$ filtration?
What might be the $PF$ martingale of reference, perhaps : $B_{[s,t]}$?
\begin{exo}
Compute :
\[\mathbb{E}\left[\int_0^\infty du\,e^{-\lambda u}\,f(B_u)|\mathcal{F}_{s,t}\right].\]
A first attempt :
\begin{eqnarray*}
\mathbb{E}\left[\int_0^\infty du\,e^{-\lambda u}\,f(B_u)|\mathcal{F}_{s,t}\right]&=&\int_0^sdu\,e^{-\lambda u}\,f(B_u)+\int_t^\infty du\,e^{-\lambda u}\,f(B_u)\\
&+&\int_s^t du\,e^{-\lambda u}\mathbb{E}\left[f(B_u)|\mathcal{F}_{s,t}\right].
\end{eqnarray*}
We observe that, for $s<u<t$ :
\[B_u=B_s+\left(\frac{u-s}{t-s}\right)(B_t-B_s)+R,\]
with $R$ independent of $(B_s, B_t-B_s)$ and $\mathbb{E}\left[R^2\right]\equiv r^2=\frac{(u-s)(t-u)}{(t-s)}$. Hence :
\begin{equation}\label{trucmuch}
\mathbb{E}\left[f(B_u)|\mathcal{F}_{s,t}\right]=\frac{1}{\sqrt{2\pi
    r^2}}\int_{-\infty}^\infty
dx\,e^{-\frac{x^2}{2r^2}}\,f(x+B_s+(u-s)B_{[s,t]}).
\end{equation}
We also note that we have obtained :
\[
 \forall y\in\mathbb{R},\quad \frac{1}{\sqrt{2\pi
     r^2}}e^{-\frac{\left(y-(B_s+(u-s)B_{[s,t]})\right)^2}{2r^2}} \]
 is a $(\mathcal{F}_{s,t})$ martingale for $s<u<t$, but we need to compute :
\[\int_s^t du\,e^{-\lambda u}\mathbb{E}\left[f(B_u)|\mathcal{F}_{s,t}\right]\]
with the help of (\ref{trucmuch})
\end{exo}

\newpage
\section{Note 7 : Which results still hold for discontinuous
  martingales?}
\subsection{No positive jumps}
In this section, $(M_t,t\geq 0)$ denotes a càdlàg
$\mathbb{R}^+$-valued local martingale, such that : $M_0=1$,
$\lim_{t\to\infty}M_t=0$, and $(M_t)$ has no positive jumps.

Then, our previous main results still hold; precisely :
\begin{proposition}
Under the previous hypotheses,
\begin{enumerate}
\item[a)]
\[\sup_{s\geq 0}M_s\overset{\underset{\mathrm{law}}{}}{=}\frac{1}{\mathbf{U}};\]
\item[b)] if $\mathcal{G}_K(M)=\sup\{s\geq 0, M_s\geq K\}$, with
  $K>0$, then, for any stopping time $T$ :
\begin{equation}
\mathbb{P}\left(\mathcal{G}_K(M)>T|\mathcal{F}_T\right)=\left(\frac{M_T}{K}\right)\wedge 1.
\end{equation}
\end{enumerate}
\end{proposition}
\underline{\textbf{Remark :}} In the same vein, we mention that the so-called ``Azéma-Yor martingales'' $F(S_t,M_t)$, with $S_t=\sup_{s\leq t} M_s$ admit an extension (with no formal change in the formula) to the family of martingales with no positive jumps. See L. Nguyen- M. Yor \cite{nguyen}.
\subsection{A formula in the Lévy process framework}
We would like to apply the preceeding to $M_t=\exp{\left(X_t-t\psi(1)\right)}$,
$t\geq 0$, where $(X_t, t\geq 0)$ is a Lévy process with no Brownian component.

A general version of Tanaka's formula ( see \cite{yor5}) is :
\begin{equation}
\left(M_t-K\right)^+=\left(M_0-K\right)^++\int_0^t1_{\{M_{s^-}>K\}}\,dM_s+\Sigma_t^{(K)},
\end{equation}
where :
\begin{equation}
\Sigma_t^{(K)}=\sum_{s\leq t}\left(1_{\{M_{s^-}>K\}}\left(M_s-K\right)^-+1_{\{M_{s^-}\leq K\}}\left(M_s-K\right)^+\right).
\end{equation}
Let $\nu(dx)$ denote the Lévy measure of $X$. Then, Lévy's
compensation formula gives :
\begin{equation}
\begin{split}
\mathbb{E}\left[\Sigma_t^{(K)}\right]&=\mathbb{E}\left[\int_0^t1_{\{M_{s^-}>K\}}\,ds\int\nu(dx)\left(M_{s^-}e^x-K\right)^-\right]\\
&+\mathbb{E}\left[\int_0^t1_{\{M_{s^-}\leq K\}}\,ds\int\nu(dx)\left(M_{s^-}e^x-K\right)^+\right]\\
&=\int_0^tds\,\phi(s,K), 
\end{split}
\end{equation}
with :
\begin{equation}
\begin{split}
\phi(s,K)&=\int_0^1 du\,\mu(]0,u[)\mathbb{E}\left[M_s\,1_{\{K<M_s<\frac{K}{u}\}}\right] \\
&+ \int_1^\infty du\,\mu([u,\infty[)\mathbb{E}\left[M_s\,1_{\{\frac{K}{u}<M_s<K\}}\right],
\end{split}
\end{equation}
where $\mu$ is the image of $\nu$ by the exponential application :
$x\to\exp{(x)}$.

Finally, we have obtained :
\[\mathbb{E}\left[\left(M_t-K\right)^+\right]=\left(M_0-K\right)^++\int_0^tds\,\phi(s,K),
\]
which plays the same role as Tanaka's formula in the continuous framework.

\newpage

\section{Note 8 : Other option prices which increase with maturity}
The contents of this note are taken entirely from a preprint by
Carr-Ewald-Xiao \cite{ewald}, to whom we are grateful for free access.

The main result of that paper is the following :
\begin{theorem}
Let $g$ be a continuous convex function; then, the function :
\[\mathcal{A}_g^{(t)}=\mathbb{E}\left[g\left(\frac{1}{t}\int_0^tds\,\mathcal{E}_s\right)\right]\]
is an increasing function of t.
\end{theorem}
\begin{exo}
Compute, for any $n\geq 1$, the function :
\[a_n(t)=\mathbb{E}\left[\left(\frac{1}{t}\int_0^tds\,\mathcal{E}_s\right)^n\right]\equiv \frac{\tilde{a}_n(t)}{t^n}.\]
\underline{Partial solution and hints :}
\begin{itemize}
\item $a_1(t)=1$;
\item $a_2(t)=\frac{2}{t^2}(e^t-1-t)$;
\item In order to compute $a_n(t)$, it may be useful to first prove the following formula : for $\alpha$ sufficiently large,
\[\int_0^\infty dt\,e^{-\alpha t}\tilde{a}_n(t)=\frac{n!}{\alpha^2(\alpha-1)(\alpha-3)\dots(\alpha-\frac{n(n-1)}{2})},\]
where 
$$\tilde{a}_n(t)=\mathbb{E}\left[\left(\int_0^t ds\,\mathcal{E}_s\right)^n\right]\equiv t^n a_n(t);$$
\item More directly, we obtain :
\end{itemize}
\begin{eqnarray*}
a_n(t)&=&\frac{n!}{t^n}\:\mathbb{E}\Big[\int_0^t ds_1\int_{s_1}^t ds_2\dots\int_{s_{n-1}}^t ds_n \\
& &\exp{\left((B_{s_1}+\dots+B_{s_n})-\frac{1}{2}(s_1+\dots+s_n)\right)}\Big]\\
&=&\frac{n!}{t^n}\int_0^t ds_1\int_{s_1}^t ds_2\dots\int_{s_{n-1}}^t ds_n\,\exp{\left(\frac{1}{2}C(s_1,\dots,s_n)\right)}
\end{eqnarray*}
where :
\[
C(s_1,\dots,s_n)=
\mathbb{E}\left[(B_{s_1}+\dots+B_{s_n})^2\right]-(s_1+\dots+s_n).\]
and we find :
\[\mathbb{E}\left[(B_{s_1}+\dots+B_{s_n})^2\right]=n^2s_1+(n-1)^2(s_2-s_1)+\dots+(s_n-s_{n-1}).\]
Hence :
\begin{eqnarray*}
C(s_1,\dots,s_n)&\equiv&\sum_{j=0}^{n-1}\left((n-j)^2-(n-j)\right)(s_{j+1}-s_j)\\
&\equiv&\sum_{j=0}^{n-1}\left((n-j)(n-(j+1))\right)(s_{j+1}-s_j)
\end{eqnarray*}
\begin{itemize}
\item Finally, it may be useful in order to solve this question to use the ``Asian Option identity in law'' :
\[A^{(\nu)}_{T_\lambda}=\int_0^{T_\lambda}ds\,\exp{2(B_s+\nu s)}\overset{\underset{\mathrm{law}}{}}{=}\frac{\beta_{1,a}}{2\gamma_b},\]
where $T_\lambda$ is an independent exponential time with parameter $\lambda$, and $\beta_{1,a}$ and $\gamma_b$ are 2 independent random variables, respectively distributed as beta$(1,a)$ and gamma$(b)$, with $a=\frac{\mu+\nu}{2}$, $b=\frac{\mu-\nu}{2}$, and $\mu=\sqrt{2\lambda+\nu^2}$. For a compendium of results/papers on this topic, see Yor \cite{yorbis}.
\end{itemize}
\end{exo}

\newpage
\section{Note 9 : From puts to calls : more care is needed!}
\subsection{Why should one be careful?}
In pricing financial options , the $LHS$ of Part A-$(12)$ arises very naturally in terms of put options, i.e. when considering :
\begin{equation}\label{9.1}
\mathbb{E}\left[\left(K-M_t\right)^+\right]
\end{equation}
On the other hand, the price of a call option is :
\begin{equation}\label{9.2}
\mathbb{E}\left[\left(M_t-K\right)^+\right]
\end{equation}
A most common argument to ``reduce'' (\ref{9.2}) to (\ref{9.1}) is to invoke ``call-put parity'' and/or ``change of numéraire''. Mathematically, this means that we consider the new probability $\mathbb{Q}$ defined via :
\begin{equation}\label{gir}
\mathbb{Q}\vert_{\mathcal{F}_t}=M_t\cdot\mathbb{P}\vert_{\mathcal{F}_t},
\end{equation}
and the martingale $(\frac{1}{M_t}, t\geq 0)$ under $\mathbb{Q}$, since :
\begin{equation}\label{9.24}
\begin{split}
\mathbb{E}_{\mathbb{P}}\left[\left(M_t-K\right)^+\right]&=\mathbb{E}_{\mathbb{Q}}\left[\left(1-\frac{K}{M_t}\right)^+\right]\\
&=K\,\mathbb{E}_{\mathbb{Q}}\left[\left(\frac{1}{K}-\frac{1}{M_t}\right)^+\right].
\end{split}
\end{equation}
However, two difficulties arise in order to perform these operations \\rigorously :
\renewcommand{\theenumi}{\roman{enumi}.}
\renewcommand{\labelenumi}{\theenumi}
\begin{enumerate}
\item in order that $\mathbb{Q}$, as defined via (\ref{gir}), be a probability, we need that $(M_t,t\geq 0)$ is a true martingale under $\mathbb{P}$, i.e : it satisfies in particular $\mathbb{E}_{\mathbb{P}}(M_t)=1$;
\item some care is needed also concerning (\ref{9.24}); in particular $M_t$ might take the value $0$ on some $\mathcal{F}_t$-set of positive $\mathbb{P}$-probability.
\end{enumerate}
To summarize, (\ref{gir}) and (\ref{9.24}) are correct\\
\underline{if $(M_t, t\geq 0)$ is a strictly positive true martingale under $\mathbb{P}$}.

Formally, this may be stated as :
\begin{proposition}\label{prop9.4}
If $(M_t)$ is a strictly positive true continuous martingale under $\mathbb{P}$, define $\mathbb{P}^M$ via :
\[\mathbb{P}^M\vert_{\mathcal{F}_t}=M_t\cdot\mathbb{P}\vert_{\mathcal{F}_t}.\]
Denote :
\[g_\infty^{(1)}=\sup\{t\geq 0, M_t=1\}.\]
Then :
\renewcommand{\theenumi}{\roman{enumi}.}
\renewcommand{\labelenumi}{\theenumi}
\begin{enumerate}
\item $\mathbb{E}_{\mathbb{P}}\left[F_t\left(M_t-1\right)^+\right]=\mathbb{E}^M\left[F_t\,1_{\{g_\infty^{(1)}\leq t\}}\right]$, for every $F_t\in\mathcal{F}_t$;
\item $\mathbb{E}_{\mathbb{P}}\left[F_t|M_t-1|\right]=\mathbb{E}_{\mathbb{P}}\left[F_t\,1_{\{g_\infty^{(1)}\leq t\}}\right]+\mathbb{E}^M\left[F_t\,1_{\{g_\infty^{(1)}\leq t\}}\right]$;
\item $g_\infty^{(1)}$ has the same distribution under $\mathbb{P}$ and under $\mathbb{P}^M$.
\end{enumerate}
\end{proposition}
\begin{proof}
\renewcommand{\theenumi}{\roman{enumi}.}
\renewcommand{\labelenumi}{\theenumi}
\begin{enumerate}
\item We write :
\begin{eqnarray*}
 \mathbb{E}_{\mathbb{P}}\left[F_t\left(M_t-1\right)^+\right]&=& \mathbb{E}^M\left[F_t\left(1-\frac{1}{M_t}\right)^+\right]
 \\
&=&\mathbb{E}^M\left[F_t\,1_{\{g_{\infty}^{(1)}\leq t\}}\right],
\end{eqnarray*}
from Theorem 1.1 (Part A), since $\left(\frac{1}{M_t},t\geq 0\right)$ is, under $\mathbb{P}^M$, a martingale which converges to $0$ as $t\to\infty$.
\item 
\[
\mathbb{E}_{\mathbb{P}}\left[F_t|M_t-1|\right]=\mathbb{E}_{\mathbb{P}}\left[F_t\left(M_t-1\right)^+\right]+\mathbb{E}_{\mathbb{P}}\left[F_t\left(1-M_t\right)^+\right],
\]
and we apply both the previous result and Theorem 1.1 (Part A).
\item Taking $F_t=1$ in i., we obtain :
\begin{eqnarray*}
\mathbb{P}^M\left(g_\infty^{(1)}\leq t\right)&=&\mathbb{E}_{\mathbb{P}}\left[\left(M_t-1\right)^+\right]\\
&=&\mathbb{E}_{\mathbb{P}}\left[\left(M_t-1\right)\right]+\mathbb{E}_{\mathbb{P}}\left[\left(M_t-1\right)^-\right]\\
&=&\mathbb{E}_{\mathbb{P}}\left[\left(1-M_t\right)^+\right]\\
&=&\mathbb{E}_{\mathbb{P}}\left[g_\infty^{(1)}\leq t\right],
\end{eqnarray*}
from Theorem 1.1 (Part A).
\end{enumerate}
\end{proof}
We note that Proposition \ref{prop9.4} extends in the case when $(M_t)$ is a true continuous martingale, taking values in $\mathbb{R}^+$, but it may vanish, i.e : $\mathbb{P}\left(T_0<\infty\right)>0$. Indeed, under $\mathbb{P}^M$, $T_0=\infty$ a.s., and the previous arguments are still valid.

In order to obtain some analogue of Theorem 1.1 (Part A) for :
\[\mathbb{E}\left[F_t\left(M_t-K\right)^+\right]\]
in the general case when $(M_t)$ is a local martingale belonging to $\mathcal{L}_0^{(+)}$, we shall proceed directly.

\subsection{The ``common'' Bessel example}
Before discussing in a general framework, we write some version 
of the identity Part A-$(12)$ for $(M_t=\frac{1}{R_t}, t\geq 0)$, where $R$ is a $BES(3)$, starting from $1$.\\
We denote by $\mathbb{W}_a$ and $\mathbb{P}_a^{(3)}$ $(a>0)$ the respective laws of Brownian motion and $BES(3)$ starting at $a>0$, on the canonical space $\mathcal{C}(\mathbb{R}^+,\mathbb{R})$, with $X_t(\omega)=\omega(t)$, and $\mathcal{F}_t=\sigma\{X_s,s\leq t\}$; then, there is the well-known Doob's $h$-transform relationship :
\begin{equation}\label{gir2}
\mathbb{P}_a^{(3)}\vert_{\mathcal{F}_t}=\frac{X_{t\wedge T_0}}{a}\,\cdot\,\mathbb{W}_a\vert_{\mathcal{F}_t}.
\end{equation}

\begin{proposition}\label{prop9.6}
There is the identity :
\begin{equation}\label{9.26}
\mathbb{E}_1^{(3)}\left[F_t\left(\frac{1}{X_t}-1\right)^+\right]=\mathbb{W}_1\left(F_t\,1_{\{\gamma\leq t\leq T_0\}}\right),
\end{equation}
for every $F_t\in\mathcal{F}_t$, and $\gamma=\sup\{t<T_0, X_t=1\}$.
\end{proposition}
\begin{proof}
Thanks to (\ref{gir2}), the $LHS$ of (\ref{9.26}) equals :
\[\mathbb{W}_1\left(F_t\left(1-X_{t\wedge T_0}\right)^+\,1_{\{t\leq T_0\}}\right),\]
which is equal to the $RHS$ of (\ref{9.26}) thanks to formula Part A-$(12)$, applied with $F'_t=F_t\,1_{\{t\leq T_0\}}$. 
\end{proof}

\textbf{Remarks.}
\begin{enumerate}
\item[a)] We understand that the present discussion is close to recent work by S. Pal and P. Protter motivated by financial bubbles. In fact, we had acces to their preprint \cite{protter2}, at the beginning of April 2008 and after writing this Note 9.
\item[b)] It is worth noting that, as a consequence of (\ref{9.26}), there is the identity 
\[\mathbb{E}_1^{(3)}\left[\left(\frac{1}{X_t}-1\right)^+\right]=\mathbb{W}_1
\left(\gamma\leq t\leq T_0\right),\]
which shows clearly that the $LHS$ is not an increasing function of $t$; indeed, the $RHS$ converges to $0$ as $t\to\infty$, as a consequence of Lebesgue's dominated convergence theorem.\\
In fact, we can compute explicitly this $RHS$, which equals :
\begin{eqnarray*}
r(t)&\equiv& \mathbb{W}_1\left(T_0\geq t\right)-\mathbb{W}_1\left(\gamma\geq t\right)\\
&\equiv&\left(1-\mathbb{W}_1\left(T_0\leq t\right)\right)-\left(1-\mathbb{W}_1\left(\gamma\leq t\right)\right)\\
&\equiv& \mathbb{W}_1\left(\gamma\leq t\right)-\mathbb{W}_1\left(T_0\leq t\right).
\end{eqnarray*}
Recall that, under $\mathbb{W}_1$ : $T_0\overset{\underset{\mathrm{law}}{}}{=}\frac{1}{B_1^2}$ and
 $\gamma\overset{\underset{\mathrm{law}}{}}{=}\frac{\mathbf{U}_{[0,2]}^2}{B_1^2}$; thus :
\begin{equation}\label{9.etoile}
r(t)=\mathbb{P}\left(|B_1|\leq \frac{1}{\sqrt{t}}\right)-\mathbb{P}\left(|B_1|\leq \frac{\mathbf{U}_{[0,2]}}{\sqrt{t}}\right),
\end{equation}
that is :
\begin{equation}\label{9.2.etoile}
r(t)= \sqrt{\frac{2}{\pi}}\int_0^{\frac{1}{\sqrt{t}}}\,dx\,e^{-x^2/2}-\sqrt{\frac{2}{\pi}}\int_0^\infty\,dx\,e^{-x^2/2}\,\left(1-\frac{\sqrt{t}x}{2}\right)^+
\end{equation}
In particular, it easily follows from (\ref{9.etoile}) that :
\begin{equation}\label{9.2.231}
r(t)\sim_{t\to\infty}\sqrt{\frac{2}{\pi t^3}}\left(\frac{1}{6}\right).
\end{equation}
It is also easily proven directly that :
\begin{equation}\label{9.2.232}
r(t)\sim_{t\to 0}\sqrt{\frac{t}{2\pi}}
\end{equation}
Formulae (\ref{9.2.231}), (\ref{9.2.232}) are proven in the Appendix.
\begin{exo}
Prove the formula :
\[\int_0^\infty \lambda e^{-\lambda t}\,dt\,r(t)=\left(\frac{1-e^{-2\sqrt{2\lambda}}}{2\sqrt{2\lambda}}\right)-e^{-\sqrt{2\lambda}}.\]
\end{exo}
\item[b)] There is an extension of formula (\ref{9.26}) for pairs of Bessel processes with respective dimensions $\delta\in(0,2)$, and $4-\delta$; formula (\ref{gir2}) generalizes as :
\[\mathbb{P}_1^{(4-\delta)}\left(F_t\,\left(\frac{1}{X_t^{2-\delta}}-1\right)^+\right)=\mathbb{P}_1^{(\delta)}\left(F_t\,1_{\{\gamma\leq t\leq T_0\}}\right),\]
for every $F_t\in\mathcal{F}_t$, and $\gamma=\sup\{t<T_0, X_t=1\}$.
Proposition \ref{prop9.6} leads us easily to a general statement, whose proof simply mimicks that of  Proposition \ref{prop9.6}, hence it is left to the reader.
\begin{proposition}\label{prop9.3}
Let $\mathbb{Q}\vert_{\mathcal{F}_t}=M_t\cdot\mathbb{P}\vert_{\mathcal{F}_t}$, with $M_0=1$, and let us assume that $T_0=\inf\{t, M_t=0\}<\infty$ $\mathbb{P}$ a.s.\\[0.1cm]
Then, $M'_t=\frac{1_{\{t<T_0\}}}{M_t}$ is well defined and strictly positive under $\mathbb{Q}$; \\$M'_t\in\mathcal{L}_0^{(+)}$; moreover it is a strict local martingale under $\mathbb{Q}$.\\
Finally :
\begin{equation}\label{9.27}
\mathbb{E}_{\mathbb{Q}}\left[F_t\left(M'_t-1\right)^+\right]=\mathbb{E}_{\mathbb{P}}\left[F_t\,1_{\{\gamma\leq t<T_0\}}\right],
\end{equation}
where $\gamma=\sup\{t<T_0, M_t=1\}$.
\end{proposition}
Note that, as for Proposition \ref{prop9.6}, the $RHS$ of (\ref{9.27}), for $F_t\equiv 1$, is no longer an increasing function of $t$.\\
The result (\ref{9.27}) should be compared with the general expression given in \cite{madyor}, Proposition 2, p.160, for $\mathbb{E}_{\mathbb{Q}}\left[\left(M'_t-K\right)^+\right]$.
\end{enumerate}
The contruction of $(M'_t)$ in Proposition \ref{prop9.3} is very close to the discussion by Delbaen-Schachermayer \cite{delbaen} of arbitrage within the Bessel processes framework.
\subsection{A general discussion}
This is based on \cite{madyor}; extensions of that discussion to discontinuous local martingales have been done by Chybiryakov \cite{chybi}, Kaji \cite{kagi}.

Let $(M_t,t\geq 0)$ denote an $\mathbb{R}^+$-valued continuous local martingale, with $M_0=a>0$. We denote by $(\mathcal{L}_t^K,t\geq 0)$ the local time at level $K$ of $(M_t,t\geq 0)$.

The following theorem shows how careful one should be when ``integrating Tanaka's formula'' when $(M_{t\wedge\tau},t\geq 0)$ is not a uniformly integrable martingale, with $\tau$ a stopping time. This care should be taken in particular for fixed time $\tau$ when $(M_t)$ is a strict local martingale.
\begin{theorem}\label{9}[\cite{madyor}, Theorem 1]
With the previous notation, and hypotheses, there is the formula :
\begin{equation}\label{77}
\mathbb{E}\left[\left(M_\tau-K\right)^+\right]=\left(a-K\right)^++\frac{1}{2}\,\mathbb{E}\left[\mathcal{L}_\tau^K\right]-c_M(\tau),
\end{equation}
where the ``correction term'' $c_M(\tau)$ does not depend on K, and is equal to either of the four following quantities :
\renewcommand{\theenumi}{\roman{enumi}.}
\renewcommand{\labelenumi}{\theenumi}
\begin{enumerate}
\item $\mathbb{E}\left[a-M_\tau\right]$;
\item $\lim_{K\to\infty}\frac{1}{2}\,\mathbb{E}\left[\mathcal{L}_\tau^K\right]$;
\item $\lim_{\sigma\to\infty} \sigma\,\mathbb{P}\left(\sup_{u\leq \tau}M_u\geq \sigma\right)$;
\item $\lim_{q\to\infty}\sqrt{\frac{\pi}{2}}\left(q\,\mathbb{P}\left(\sqrt{\langle M\rangle_\tau}\geq q\right)\right)$.
\end{enumerate}
\end{theorem}
See \cite{madyor} for a discussion of the financial meaning of (\ref{77}). Let us insist, once again, about the formula different, from (\ref{77}), which one obtains when considering the put quantity : $\mathbb{E}\left[\left(K-M_\tau\right)^+\right]$, instead of the call quantity : $\mathbb{E}\left[\left(M_\tau-K\right)^+\right]$ :
\begin{equation}\label{78}
\mathbb{E}\left[\left(K-M_\tau\right)^+\right]=\left(K-a\right)^++\frac{1}{2}\,\mathbb{E}\left[\mathcal{L}_\tau^K\right].
\end{equation}
Clearly, there is agreement between the identities (\ref{77}),(\ref{78}) and the expression $(i)$ of $c_M(\tau)$.
\newpage
\section{Note 10 : A temporary conclusion}
\subsection{What have we learnt?}
\begin{enumerate}
\item[a)]Clearly, the knowledge of the call/put quantities : \[\mathbb{E}\left[\left(M_t-K\right)^{\pm}\right],\] 
for all strikes and maturities, is equivalent to the knowledge of the 1-dimensional marginals of $M$. We have learnt that it is also equivalent to the knowledge of the laws of last passage times $\mathcal{G}_K$. 
\item[b)] In our discussion, some fundamental quantities are certain conditional expectations, namely :
\[\mathbb{E}\left[M_\infty|S_\infty\right]\:\:\mathrm{and}\:\:\mathbb{E}\left[\sigma_s^2|M_s\right],\]
which would certainly deserve more study. However, once introduced, they allow to develop our discussion in a purely martingale framework instead of the more usual Markovian framework.
\end{enumerate}
\subsection{Where is all this going?}
This course seems to open a number of fairly general questions :
\begin{enumerate}
\item[a)] (Note 1). May the universal character of the law of $\sup_{u\geq t}M_u$ be extended to a study involving :
\[\int_t^\infty d\langle M\rangle_u\,f(M_u),\:\mathrm{for\:suitable}\:f\:\mathrm{'s?}\] 
\item[b)] (Note 6). Can one develop a stochastic calculus with respect to the two-parameter filtration $(\mathcal{F}_{s,t})$? Creating a two-parameter Itô stochastic integration has already been done, in the context of the study of double points of 2- and 3-dimensional Brownian motion by J. Rosen, M. Yor \cite{rosen}. We would also like to connect this study with the fairly well developed Markovian Fields theory. (see, e.g, Dang-Ngoc-Yor \cite{dang}, Royer-Yor \cite{royer}).
\item[c)] (Note 8 )The Black-Scholes formula has been the first fundamental option pricing formula, and may be considered to have fathered many other such formulae, e.g : for quantile options, Asian options, and so on... Could one develop a last time or more generally a past-future viewpoint for these options?
\item[d)] We plan to develop answers to some of these questions in \cite{amel}.
\end{enumerate}
\subsection{Dealing with last passage times :}
Throughout both Part A and Part B, last passage times (of local martingales) play a crucial role. It has been comforting that in private discussions around these lectures we did not experience. The reluctance about these ``non-stopping times'', which existed during the seventies, as K. L. Chung \cite{chung} describes very well : ``For some reason, the notion of a last exit time would not be dealt with openly and directly. This may be due to the fact that such a time is not an ``optional'', or ``stopping time'', does not belong to the standard equipment, and so must be evaded at all costs.''

The positive attitude (which we experienced in these private discussions) may be partly be due to the fact that the main results of enlargments of filtrations theory  are nowadays reasonably well-known.

On the other hand, one should not deny that, when dealing with last passage times, difficulties occur, due to the fact that these times look forever into the future. Undoubtedly, J. Akahori's suggestion (Note 6) originated from these diffculties.
\newpage

\section{Appendix : Developments of some particular points-solutions to the exercises for both Part A and Part B}

This appendix is decomposed into sections 11-1, 11-2, ... which correspond to the Notes 1, 2, ...
\subsection{Note 1}
\begin{itemize}
\item \textbf{Solution to Exercise 1.1 :}\\
From the strong Markov property :
\begin{eqnarray*}
G_a^{(\nu)}&=&\sup\{t, B_t^{(\nu)}=a\}\\
&=&T_a^{(\nu)}+\sup\{u, B_{T_a^{(\nu)}+u}^{(\nu)}-a=0\}\\
&=&T_a^{(\nu)}+\tilde{G}_0^{(\nu)},
\end{eqnarray*}
where : $\tilde{G}_0^{(\nu)}=\sup\{u, \tilde{B}_u+\nu u=0\}$, and $(\tilde{B}_u,u\geq 0)$ is a Brownian motion starting from $0$ and independent from $(B_t)$.
Now, the law of $\tilde{G}_0^{(\nu)}$ has a simple expression, since :
\begin{eqnarray*}
\tilde{G}_0^{(\nu)}&=&\sup\{u, \tilde{B}_u=-\nu u\}\\
&\overset{\underset{\mathrm{law}}{}}{=}&\frac{1}{T^{(\nu)}(\tilde{B})}\:\:\mathrm{(:\:time\:inversion)}\\
&\overset{\underset{\mathrm{law}}{}}{=}&\frac{1}{\nu^2\tilde{T}_1}\:\:\mathrm{(:\:scaling)}\\
&\overset{\underset{\mathrm{law}}{}}{=}&\frac{\tilde{B}_1^2}{\nu^2}\:\:\mathrm{(:\:classical\:result)}.
\end{eqnarray*}
Hence :
\[\left(T_a^{(\nu)},G_a^{(\nu)}\right)\overset{\underset{\mathrm{law}}{}}{=}\left(T_a^{(\nu)}, T_a^{(\nu)}+\frac{\tilde{B}_1^2}{\nu^2}\right),\]
with $\tilde{B}_1$ independent from $T_a^{(\nu)}$.\\
\item \textbf{Solution to Exercise 1.2 :}\\
We first have thanks to Theorem $A-1.1$ :
\[
\mathbb{E}\left[\left(\mathcal{E}_t-K\right)^+\right]=K\mathbb{P}\left(0<G_{\ln{K}}^{(-1/2)}\leq t\right).
\]
The laws of $G_{a}^{(\nu)}$ and $G_{a}^{(-\nu)}$ are related via the Cameron-Martin formula :
\begin{equation}\label{appendixetoile}
\mathbb{W}^{(-\nu)}\vert_{\mathcal{F}_{G_a}\cap(G_a>0)}=\exp{(-2\nu a)}\mathbb{W}^{(\nu)}\vert_{\mathcal{F}_{G_a}}.
\end{equation}
Hence :
\[\mathbb{P}\left(G_a^{(-\nu)}\in dt|G_a^{(-\nu)}>0\right)=\mathbb{P}\left(G_a^{(\nu)}\in dt\right).\]
Moreover thanks to (\ref{appendixetoile}) :
\[\mathbb{P}\left(G_a^{(-\nu)}>0\right)=\exp{(-2\nu a)}.\]
So the result :
\[
\mathbb{E}\left[\left(\mathcal{E}_t-K\right)^+\right]=\mathbb{P}\left(G_{\ln{K}}^{(1/2)}\leq t\right),
\]
follows.
We leave the end of the proof to the reader (take derivative in $t$ of the $RHS$ of $(A-5)$ and recover the density of the law of $G_{\ln{K}}^{(1/2)}$ given in $(A-22)$.\\
\item \textbf{Solution to Exercise 1.3 :}\\
We recall that :
\[\mathbb{E}\left[\left(\mathcal{E}_t-K\right)^\pm\right]=\left(1-K\right)^\pm+\frac{1}{2}\mathbb{E}\left[\mathcal{L}_t^K(\mathcal{E})\right].\]
From the time occupation formula :
\[\mathbb{E}\left[\int_0^tf(\mathcal{E}_s)\,d\langle \mathcal{E}\rangle_s\right]
=\mathbb{E}\left[\int_0^tf(\mathcal{E}_s)\,\mathcal{E}_s^2\,ds\right]
=\mathbb{E}\left[\int_0^\infty f(K)\mathcal{L}_t^K(\mathcal{E})\,dK\right],\]
and the change of variable $e^{y-\frac{s}{2}}=K$, we deduce :
\[
\int_0^t \frac{ds}{\sqrt{2\pi s}}\int_0^\infty f(e^{y-\frac{s}{2}})\,e^{2y-s}\,e^{-\frac{y^2}{2s}}\,dy
=\int_0^t \frac{ds}{\sqrt{2\pi s}}\int_0^\infty f(K)\,K\,e^{-\frac{1}{2s}\left(\log{K}+\frac{s}{2}\right)^2}\,dK,
\]
so that :
\begin{eqnarray*}
\mathbb{E}\left[\mathcal{L}_t^K(\mathcal{E})\right]&=&\int_0^t \frac{ds}{\sqrt{2\pi s}}\,K\,e^{-\frac{1}{2s}\left(\log{K}+\frac{s}{2}\right)^2}\\
&=&\sqrt{K}\int_0^t \frac{ds}{\sqrt{2\pi s}}\,e^{-\frac{(\log{K})^2}{2s}-\frac{s}{8}}\\
&=&2\sqrt{K}\mathbb{E}\left[1_{\{4B_1^2\leq t\}}\exp{-\frac{(\log{K})^2}{8B_1^2}}\right]
\end{eqnarray*}
since the density of the r.v $B_1^2$ is : $\frac{ds}{\sqrt{2\pi s}}e^{-\frac{s}{2}}1_{[0,\infty[}(s)$.
\end{itemize}

\subsection{Note 2}
\begin{itemize}
\item \textbf{Discussion of Example 2.1 :}\\
Thanks to time reversal, one has :
\[(B_{T_0-u},u\leq T_0)\overset{\underset{\mathrm{law}}{}}{=}(R_u,u\leq G_1(R)),\]
where $(R_u,u\geq 0)$ is a $BES_0(3)$ and $G_1(R)=\inf\{u\geq 0, R_u=1\}$.
Hence :
\[T_0-G_1(M)\overset{\underset{\mathrm{law}}{}}{=}T_K(R),\]
where on the $LHS$, $T_0$ and $G_1(M)$ are independent and on the $RHS$, $T_K(R)=\inf\{u\geq 0, R_u=K\}$. Taking the Laplace transform in $(\frac{\lambda^2}{2})$ of both sides, one obtains :
\begin{equation}\label{laplace1}
e^{-\lambda}=\frac{\lambda K}{\sinh{\lambda K}}\mathbb{E}\left[e^{-\frac{\lambda^2}{2}G_K(M)}\right],
\end{equation}
since :
\[\mathbb{E}\left[e^{-\frac{\lambda^2}{2}T_0}\right]=e^{-\lambda}; \quad \mathbb{E}\left[e^{-\frac{\lambda^2}{2}G_K(R)}\right]=\frac{\lambda K}{\sinh{\lambda K}}.\]
 Formula (\ref{laplace1}) becomes :
\begin{eqnarray*}
\mathbb{E}\left[e^{-\frac{\lambda^2}{2}G_K(M)}\right]&=&\frac{e^{-\lambda(1-K)}-e^{-\lambda(1+K)}}{2\lambda K}\\
&=& \frac{1}{2K}\int_{1-K}^{1+K}\exp{(-\lambda x)}\,dx\\
&=& \frac{1}{2K}\int_{1-K}^{1+K} \mathbb{E}\left[\exp{-\frac{\lambda^2}{2}T_x}\right],
\end{eqnarray*}
where $T_x$ is the first hitting time at level $x$ of a Brownian motion $(\beta_t,t\geq 0)$ starting from $0$; thus :
\begin{eqnarray*}
\mathbb{E}\left[e^{-\frac{\lambda^2}{2}G_K(M)}\right]&=&\frac{1}{2K}\int_{1-K}^{1+K} \mathbb{E}\left[\exp{-\frac{\lambda^2x^2}{2\beta_1^2}}\right]\\
&=& \mathbb{E}\left[\exp{-\frac{\lambda^2}{2}\frac{\mathbf{U}_K^2}{\beta_1^2}}\right],
\end{eqnarray*}
since :
\[T_x\overset{\underset{\mathrm{law}}{}}{=}\frac{x^2}{\beta_1^2}.\]
\item \textbf{Discussion of Example 2.1 :}\\
We observe that :
\[\sup\{t\geq 0, \frac{1}{R_t}=K\}=\sup\{t\geq 0, R_t=\frac{1}{K}\}.\]
We consider the process $R$ as obtained by time reversal from a Brownian motion $(B_t,t\geq 0)$ starting from $\frac{1}{K}$ and killed when it first hits $0$.
Hence, with the same notation as before, we have :
\[T_0\overset{\underset{\mathrm{law}}{}}{=}G_K(M)+T_1(R),\quad\mathrm{(with\:independence)}.\]
Thus, taking once again the Laplace transform in $(\frac{\lambda^2}{2})$ of both sides, one obtains :
\begin{eqnarray*}
e^{-\lambda\frac{1}{K}}&=&\frac{\lambda}{\sinh{\lambda}}\mathbb{E}\left[e^{-\frac{\lambda^2}{2}G_K(M)}\right],\\
\mathbb{E}\left[e^{-\frac{\lambda^2}{2}G_K(M)}\right]&=&\frac{1}{2\lambda}\left(\exp{\left(-\lambda(\frac{1}{K}-1)\right)}-\exp{\left(-\lambda(\frac{1}{K}+1)\right)}\right)\\
&=&\frac{1}{2}\int_{\frac{1}{K}-1}^{\frac{1}{K}+1}e^{-\lambda x}\,dx.
\end{eqnarray*}
We leave the end of the proof to the reader (conclude as above).\bigskip
\item\textbf{Solution to Exercise 2.4 :}\\
From Itô's formula, one obtains :
\[\sigma_t=\sinh{(B_t)}e^{-\frac{t}{2}}.\]
Hence : $\sigma_t^2=M_t^2-e^{-t}$; therefore, using Theorem $A-2.1$ :
\[
\mathbb{P}\left(\mathcal{G}_K\in dt\right)=\left(1-\frac{1}{K}\right)^{+}\epsilon_0(dt)+\frac{1_{\{t>0\}}}{2K}\,\theta_t(K)\,m_t(K)\,dt.
\]
We leave the computation of $m_t(K)$ to the reader.
\end{itemize}
\setcounter{subsection}{5}
\subsection{Note 6}
\begin{itemize}
\item \textbf{Solution to Exercise 6.2 :}\\
We recall the generating function expansion of the Hermite polynomials : 
\[\exp{\left(\lambda x-\frac{\lambda^2 u}{2}\right)}=\sum_{n=0}^\infty \frac{\lambda^n}{n!}H_n(x,u).\]
Now we write :
\begin{eqnarray*}
& &\exp{\left(al-\frac{dl^2}{2}\right)}\exp{\left(b\nu-\frac{c\nu^2}{2}\right)}\exp{(fl\nu)}\\
&=&\left(\sum_{k=0}^\infty \frac{l^k}{k!}H_k(a,d)\right)+ \left(\sum_{j=0}^\infty \frac{\nu^j}{j!}H_j(b,c)\right)\left(\sum_{m=0}^\infty \frac{(l\nu)^m}{m!}f^m\right)\\
&=&\sum_{k,j,m=0}^\infty \frac{l^{k+m}\nu^{j+m}}{k!j!m!}H_k(a,d)H_j(b,c)f^m.
\end{eqnarray*}
Now, we may write :
\begin{eqnarray*}
& &\sum_{p,q=0}^\infty
l^p\nu^q\mathcal{H}_{p,q}(a,b,c,d,f)\\
&=&\sum_{k,j,n=0}^\infty
\frac{l^{k+m}\nu^{j+m}}{k!j!m!}H_k(a,d)H_j(b,c)f^m\\
&=&\sum_{m=0}^\infty \frac{f^m}{m!}\left(\sum_{p=m}^\infty
  \frac{l^p}{(p-m)!}H_{p-m}(a,d)\right)\left(\sum_{q=m}^\infty \frac{\nu^q}{(q-m)!}H_{q-m}(b,c)\right) \\
&=&\sum_{p,q=0}^\infty l^p\nu^q\left(\sum_{m\leq p,q} \frac{f^m}{m!}H_{p-m}(a,d)H_{q-m}(b,c)\frac{1}{(p-m)!(q-m)!}\right)
\end{eqnarray*}
Consequently, we have :
\[\mathcal{H}_{p,q}(a,b,c,d,f)=\sum_{m\leq p,q} H_{p-m}(a,d)H_{q-m}(b,c)\frac{1}{(p-m)!(q-m)!}.\]
\item \textbf{Solution to Exercise 6.1 :}\\
First, we need to write :
\begin{eqnarray*}
& &-\frac{2}{(t-s)}(x+\nu s-l)(y+\nu t-l)\\
&=&-\frac{2}{t-s}\left[(x+\nu
  s)(y+\nu t)-l(x+y+\nu(s+t))+l^2\right] \\
&=&-\frac{2}{t-s}\left[xy+\nu(sy+tx)+\nu^2(st)-l(x+y)-l\nu(s+t)+l^2\right]\\
&\equiv&al-\frac{dl^2}{2}+b\nu-\frac{c\nu^2}{2}+fl\nu.
\end{eqnarray*}
Comparing the two developments as polynomials in $(l,\nu)$, we obtain :
\begin{eqnarray*}
a&=&\frac{2(x+y)}{t-s};\quad b=-\frac{2}{t-s}(sy+tx)\\
c&=&\frac{4(st)}{t-s};\quad d=\frac{4}{t-s}; f=\frac{2(s+t)}{t-s}.
\end{eqnarray*}

Consequently, we obtain :
\[
H_{p,q}(s,t;x,y)=\mathcal{H}_{p,q}\left(\frac{2(x+y)}{t-s},\frac{-2(sy+tx)}{t-s},\frac{4st}{t-s},\frac{4}{t-s},\frac{2(s+t)}{t-s}\right).
\]
\end{itemize}
\setcounter{subsection}{8}
\subsection{Note 9}
\begin{itemize}
\item \textbf{Asymptotic study of $r(t)$, $t\to 0$, and $t\to\infty$.}
\begin{enumerate}
\item[a)] ($t\to 0$) : We use 
\[
r(t)=\mathbb{E}_1^{(3)}\left[\left(\frac{1}{X_t}-1\right)^+\right]\underset{t\to 0}{\sim}\mathbb{E}_1^{(3)}\left[\left(1-X_t\right)^+\right].
\]
Now, using the decomposition : 
\[X_t=1+\beta_t+\int_0^t\frac{ds}{X_s},\]
with $(\beta_t)$ a one-dimensional Brownian motion starting from $0$, one obtains :
\[r(t)\underset{t\to 0}{\sim}\mathbb{E}_1^{(3)}\left[\left(-\beta_t-\int_0^t\frac{ds}{X_s}\right)^+\right],\]
and, it is now easily seen that :
\begin{eqnarray*}
r(t)&\sim&\mathbb{E}\left[\left(-\beta_t\right)^+\right]\\
&\sim& \sqrt{t}\left(\frac{1}{2}\right)\mathbb{E}\left[|\mathbf{N}|\right]\\
&\sim& \sqrt{\frac{t}{2\pi}}
\end{eqnarray*} 
\item[b)] ($t\to\infty$) : From (\ref{9.etoile}), we obtain :
\begin{eqnarray*}
r(t)&=&\mathbb{P}\left(|B_1|\leq \frac{1}{\sqrt{t}}\right)-\mathbb{P}\left(|B_1|\leq \frac{2\mathbf{U}}{\sqrt{t}}\right)\\
&=&\sqrt{\frac{2}{\pi}}\left(\int_0^{\frac{1}{\sqrt{t}}} dx\,e^{-\frac{x^2}{2}}-\mathbb{E}\left[\int_0^{\frac{2\mathbf{U}}{\sqrt{t}}} dx\,e^{-\frac{x^2}{2}}\right]\right)\\
&\sim&\sqrt{\frac{2}{\pi t}}\left(\int_0^1 dy\,e^{-\frac{y^2}{2t}}-\mathbb{E}\left[2\mathbf{U}\int_0^1 dy\,e^{-2\frac{y^2\mathbf{U}^2}{t}}\right]\right)\\
&=&\sqrt{\frac{2}{\pi t}}\left(\int_0^1 dy\,\left(e^{-\frac{y^2}{2t}}-1\right)-\mathbb{E}\left[2\mathbf{U}\int_0^1 dy\,\left(e^{-2\frac{y^2\mathbf{U}^2}{t}}-1\right)\right]\right)\\
&\sim&\sqrt{\frac{2}{\pi t}}\left(\int_0^1 dy\,\frac{-y^2}{2t}-\mathbb{E}\left[(2\mathbf{U})\int_0^1 dy\,\left(-\frac{2y^2\mathbf{U}^2}{t}\right)\right]\right) \\
\end{eqnarray*}
Hence :
\begin{eqnarray*}
r(t)&\sim& \sqrt{\frac{2}{\pi t}}\left(\mathbb{E}\left[\frac{4\mathbf{U}^3}{t}\int_0^1 dy\,y^2\right]-\frac{1}{2t}\left(\frac{1}{3}\right)\right)\\
&=&\sqrt{\frac{2}{\pi t^3}}\left(\mathbb{E}\left[4\mathbf{U}^3\right]\left(\frac{1}{3}\right)-\frac{1}{6}\right)\\
&=&\sqrt{\frac{2}{\pi t^3}}\left(\frac{1}{6}\right)
\end{eqnarray*}
From the representation (\ref{9.etoile}) of $r(t)$ :
\[r(t)=\mathbb{P}\left(|B_1|\leq \frac{1}{\sqrt{t}}\right)-\mathbb{P}\left(|B_1|\leq \frac{2\mathbf{U}}{\sqrt{t}}\right),\] 
we can show :
\[r(t)=\sqrt{\frac{2}{\pi t}}\sum_{n=1}^\infty \frac{r_n}{t^n},\]
with : 
\[r_n=\frac{(-1)^n}{n!(2n+1)}\left(\frac{1}{2^n}-\frac{2^n}{n+1}\right).\]
(Note : $r_1=\frac{1}{6}$, $r_2=-\frac{13}{120}$).
\item[c)] We recall that :
\[r(t)\equiv \mathbb{W}_1\left(T_0\geq t\right)-\mathbb{W}_1\left(\gamma\geq t\right) .\]
Let $\theta_0$ and $\theta_1$ denote the respective densities of $T_0$ and $\gamma$; then : 
\[r'(t)=-\theta_0(t)+\theta_1(t),\]
hence, $r'(t)=0$ if and only if : $\theta_0(t)=\theta_1(t)$. To be continued...
\item[d)] \underline{A simple integral formula and a graph.}\\
From (\ref{9.etoile}), it is easy enough to obtain :
\[r(t)=\sqrt{\frac{2}{\pi}}\rho(\frac{1}{\sqrt{t}}),\]
with :
\[\rho(u)=\frac{1}{2u}\left(1-e^{-2u^2}\right)-\int_u^{2u} dx\,e^{-\frac{x^2}{2}}.\]
We now show the graph of $\rho$, for which we thank G. Pagès.
\begin{center}
\begin{figure}\caption{\underline{Graph of $\rho$}} 
\includegraphics[scale=0.7]{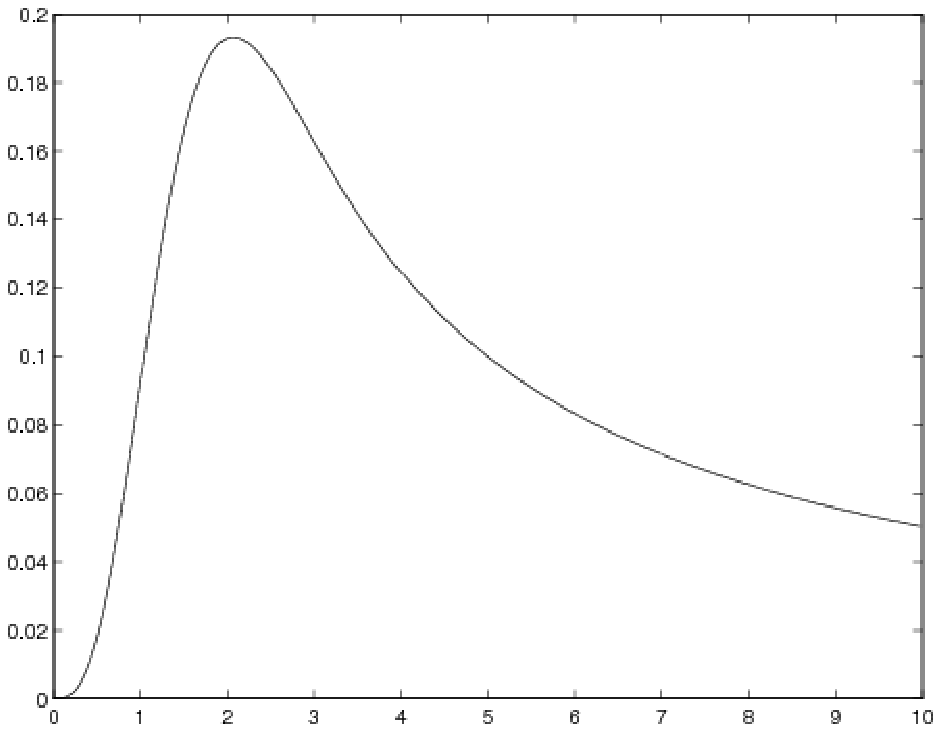}
\end{figure}
\end{center}
\item[e)] \underline{Other properties of $r$ :}
\begin{itemize}
\item We note that $(r(t),t\geq 0)$ is integrable over $\mathbb{R}^+$,
  since : $r(t)\sim \frac{C}{t^{3/2}}$, as $t\to\infty$. More precisely, we
  have :
\[\int_0^\infty dt\,r(t)=\int_0^\infty dt\,\mathbb{W}_1(\gamma\leq t\leq
T_0)=
\mathbb{W}_1(T_0-\gamma)=\mathbb{E}_0^{(3)}(T_1)=\frac{1}{3}.\]
\item Moreover, $\int_0^\infty dt\,t^{\alpha}r(t)<\infty$ iff
  $-\frac{3}{2}<\alpha<\frac{1}{2}$.
\item We note that $\tilde{r}(t)=3r(t)$, is a probability density, and we
  denote by $\tilde{R}$ a random variable such that:
\[\mathbb{P}(\tilde{R}\in dt)=3r(t)\,dt.\]
The Laplace transform of $\tilde{R}$ is given by :
\begin{eqnarray*}
\mathbb{E}\left[e^{-\lambda \tilde{R}}\right]&=&\int_0^\infty dt\,
3e^{-\lambda t}\:\mathbb{W}_1(\gamma\leq t\leq T_0)\\
&=&\frac{3}{\lambda}\:\mathbb{W}_1\left(e^{-\lambda \gamma}-e^{-\lambda
    T_0}\right)\\
&=&\frac{3}{\lambda}\:\mathbb{W}_1\left(e^{-\lambda \gamma}\right)\mathbb{W}_1\left(1-e^{-\lambda(T_0-\gamma)}\right)
\end{eqnarray*}
Since : 
\[\mathbb{W}_1\left(e^{-\lambda(T_0-\gamma)}\right)=\mathbb{E}_0^{(3)}\left[e^{-\lambda
  T_1}\right]=\frac{\sqrt{2\lambda}}{\sinh{(\sqrt{2\lambda})}},\]
and :
\[\mathbb{W}_1\left(e^{-\lambda
    \gamma}\right)\mathbb{W}_1\left(e^{-\lambda(T_0-\gamma)}\right)=e^{-\sqrt{2\lambda}},\]
we finally obtain :
\[
\mathbb{E}\left[e^{-\lambda \tilde{R}}\right]=\frac{3\left(1-e^{-2\sqrt{2\lambda}}-2\sqrt{2\lambda}e^{-\sqrt{2\lambda}}\right)}{(2\lambda)^{3/2}}.
\]
Can we exhibit a ``natural'' Brownian variable which is distributed as $\tilde{R}$?
\end{itemize}
\end{enumerate}
\end{itemize}
\newpage


\begin{center}
\huge{\textbf{\underline{Acknowledgments}}}
\end{center}

What started it all is M. Qian's question (15/08/2007).

M. Yor is also very grateful to D. Madan and B. Roynette for several
attempts to solve various questions, rewriting, and summarizing...

J. Akahori gave the stimulation for Section 6.

C. Ewald sent the preprint \cite{ewald} very early on.

The graph of $\rho$ (see A-9) is due to G. Pagès.

Lecturing in Osaka and Ritsumeikan (October 2007), Melbourne and Sydney (December 2007), then finally at the Bachelier Séminaire (February 2008) has been a
great help.

We gave further lectures in Oxford and at Imperial College, London, both in May 2008 and in Amsterdam at VU University in June 2008.

\end{document}